\documentclass[12pt]{article}

\usepackage[margin=1in]{geometry}

\usepackage{comment}
\usepackage{amsthm}
\usepackage{amssymb}
\usepackage{amsmath}
\usepackage{enumitem}
\usepackage{graphicx}
\usepackage{mathdots}
\usepackage{mathtools}
\usepackage{url}
\usepackage{color}
\usepackage{framed}
\usepackage{tikz}
\usepackage{tikz-cd}
\usepackage{stmaryrd}
\usepackage[ruled,vlined]{algorithm2e}
\usepackage{array}
\usepackage[pdfpagelabels, colorlinks = true, linkcolor = blue, citecolor = blue, urlcolor=blue, pdfdisplaydoctitle]{hyperref}
\usepackage{cleveref}
\usepackage{mathtools}
\newcolumntype{L}{>{$}l<{$}} % math-mode version of "l" column type

\newcommand{\llangle}{\langle\hspace{-2.5pt}\langle}

\newcommand{\Oname}{\operatorname{O}}

\newcommand{\im}{\operatorname{im}}

\newcommand{\rrangle}{\rangle\hspace{-2.5pt}\rangle}

\newlength{\claimstepvspace} % Declare a new length variable
\newlength{\theoremstepvspace} % Declare a new length variable
\newtheoremstyle{mystyle}%                % Name
  {}%                                     % Space above
  {}%                                     % Space below
  {}%\itshape                                     % Body font
  {}%                                     % Indent amount
  {\bfseries}%                            % Theorem head font
  {.}%                                    % Punctuation after theorem head
  { }%                                    % Space after theorem head, ' ', or \newline
  {\thmname{#1}\thmnumber{ #2}\thmnote{ (#3)}}%                                     % Theorem head spec (can be left empty, meaning `normal')

\newtheorem{theorem}{Theorem}%[section]
\newtheorem{corollary}[theorem]{Corollary}
\newtheorem{lemma}[theorem]{Lemma}
\newtheorem{proposition}[theorem]{Proposition}
\newtheorem{remark}[theorem]{Remark}

\theoremstyle{definition}
\newtheorem{problem}[theorem]{Problem}
\newtheorem{example}[theorem]{Example}
\newtheorem{definition}[theorem]{Definition}

\newcounter{claimcounter}
\theoremstyle{mystyle}
\newtheorem{claim}[claimcounter]{Claim}

\crefname{claim}{Claim}{Claims}
\crefname{figure}{Figure}{Figures}

\crefname{lemma}{Lemma}{Lemmas}
\crefname{example}{Example}{Examples}
\crefname{definition}{Definition}{Definitions}
\crefname{remark}{Remark}{Remarks}
\crefname{theorem}{Theorem}{Theorems}
\crefname{corollary}{Corollary}{Corollaries}
\crefname{proposition}{Proposition}{Propositions}
\crefname{section}{Section}{Sections}
\crefname{appendix}{Appendix}{Appendices}
\crefname{problem}{Problem}{Problems}
\crefname{claimcounter}{Claim}{Claims}

\newlist{propenum}{enumerate}{1} % also creates a counter called 'propenumi'
\setlist[propenum]{label=(\alph*), ref=\theproposition(\alph*)}
\crefalias{propenumi}{proposition} 

\newlist{claimenum}{enumerate}{1} % also creates a counter called 'propenumi'
\setlist[claimenum]{label=(\alph*), ref=\theclaim(\alph*)}
\crefalias{claimenumi}{claim}

\newlist{lemenum}{enumerate}{1} % also creates a counter called 'propenumi'
\setlist[lemenum]{label=(\alph*), ref=\thelemma(\alph*)}
\crefalias{lemenumi}{lemma} 

\newlist{thmenum}{enumerate}{1} % also creates a counter called 'propenumi'
\setlist[thmenum]{label=(\alph*), ref=\thetheorem(\alph*)}
\crefalias{thmenumi}{theorem} 

\newlist{corenum}{enumerate}{1} % also creates a counter called 'propenumi'
\setlist[corenum]{label=(\alph*), ref=\thecorollary(\alph*)}
\crefalias{corenumi}{corollary} 

\title{A max filtering local stability theorem with application to weighted phase retrieval and cryo-EM}
\author{Yousef~Qaddura\footnote{Department of Mathematics, The Ohio State University, Columbus, OH}
}

\date{}

\begin{document}
\maketitle
\begin{abstract}
    Given an inner product space $V$ and a group $G$ of linear isometries, max filtering offers a rich class of convex $G$-invariant maps. In this paper, we identify sufficient conditions under which these maps are locally bilipschitz on $R(G)$, the set of orbits with maximal dimension, with respect to the quotient metric on the orbit space $V/G$. Central to our proof is a desingularization theorem, which applies to open, dense neighborhoods around each orbit in $R(G)/G$ and may be of independent interest.

    As an application, we provide guarantees for stable weighted phase retrieval. That is, we construct componentwise convex bilipschitz embeddings of weighted complex (resp.\ quaternionic) projective spaces. These spaces arise as quotients of direct sums of nontrivial unitary irreducible complex (resp.\ quaternionic) representations of the group of unit complex numbers $S^1\cong \operatorname{SO}(2)$ (resp.\ unit quaternions $S^3\cong \operatorname{SU}(2)$).

We also discuss the relevance of such embeddings to a nearest-neighbor problem in single-particle cryogenic electron microscopy (cryo-EM), a leading technique for resolving the spatial structure of biological molecules.

\end{abstract}

\section{Introduction}

Machine learning algorithms are often designed for Euclidean data, represented as vectors in an inner product space $V$. For instance, fast randomized nearest neighbor algorithms, such as the one in \cite{JonesOR:11}, efficiently approximate nearest neighbors in large Euclidean datasets by avoiding explicit computation of all pairwise distances.

However, many practical data representations in $V$ involve ambiguities arising from an orthogonal symmetry group $G\leq \Oname(V)$. For example, as discussed in \cite{ZhaoS:14,ZhaoSS:16}, data from cryogenic electron microscopy (cryo-EM) may reside in a finite-dimensional complex vector space $\mathbb C^d$, subject to an ambiguity induced by a diagonal circle action $S^1\to \mathbb C^{d\times d}$ defined by $\theta \mapsto \operatorname{diag}\{e^{\mathrm ik_j\theta}\}_{j=1}^d$, where $\{k_{j}\}_{j=1}^d$ are arbitrary fixed integers. This example is elaborated in \cref{sec.relevance of results}. (Throughout the paper, we view $\mathbb C^d$ as a real vector space with inner product $\langle z,x\rangle := \operatorname{Re}(z^*x)$. In particular, the aforementioned action is indeed real-orthogonal.)

To address such ambiguities, one represents the data unambiguously as orbits
$[x]:= G\cdot x$ in the quotient space $V/G$, equipped with the quotient metric
\[d([x],[y]):= \inf_{\substack{p\in [x]\\ q\in [y]}}\|p-q\|.\]
(Indeed, this metric is nondegenerate provided the $G$-orbits are topologically closed). To leverage the extensive machinery of Euclidean-based machine learning, it is desirable to embed the orbit space into Euclidean space in a \textbf{bilipschitz} manner. Specifically, we seek a map $f\colon V/G\to \mathbb R^n$ and constants $\alpha,\beta > 0$ such that
\[\alpha \cdot d([x],[y]) \leq \|f([x])-f([y])\|\leq \beta\cdot d([x],[y]) \qquad \forall x,y\in V.\]
The bilipschitz requirement ensures that distances in $V/G$ are faithfully preserved, enabling robust transfer of Euclidean algorithms to the orbit space. For example, adapting the $\lambda$-approximate nearest neighbor problem to $V/G$ becomes straightforward when such embeddings are available.

\begin{example}[Example~1 in~\cite{CahillIM:24}]
  \label{ex.bilip nearest}
  Given $\lambda\geq 1$ and data $[x_1],\dots,[x_m]\in V/G$, the \textbf{$\lambda$-approximate nearest neighbor problem} takes as input $[x]\in V/G$ and outputs $j\in \{1,\dots,m\}$ such that
  \[d([x],[x_j])\leq \lambda\cdot \min_{1\leq i\leq m} d([x],[x_i]).\]
  Given a map $f\colon V/G\to \mathbb R^n$ with bilipschitz bounds $\alpha,\beta\in (0,\infty)$ and given a black box algorithm that solves the problem in $\mathbb R^n$, one may transfer the algorithm to $V/G$ by pulling back through $f$. This results in a solution of the $\frac{\beta}{\alpha}\lambda$-approximate nearest neighbor problem in $V/G$. To see this, first use the black box algorithm to find $j\in \{1,\dots,m\}$ such that
  \[\|f([x])-f([x_j])\|\leq \lambda\cdot \min_{1\leq i\leq m} \|f([x])-f([x_i])\|.\]
  Then 
  \[d([x],[x_j])\leq\frac{1}{\alpha}\cdot \|f([x])-f([x_j])\|
  \leq\frac{\lambda}{\alpha}\cdot\min_{i\in I}\|f([x])-f([x_i])\|
  \leq \frac{\beta}{\alpha} \lambda\cdot\min_{i\in I} d([x],[x_i]).
  \]
\end{example}

To this end, \cite{CahillIMP:22} recently introduced a family of embeddings called \textit{max filter banks} that enjoy explicit bilipschitz bounds whenever $G$ is finite. Later work improved on those bounds~\cite{MixonP:22,MixonQ:22}. 

\begin{definition}
Consider any real inner product space $V$ and $G\leq\operatorname{O}(V)$.
\begin{itemize}
\item[(a)]
The \textbf{max filtering map} $\llangle\cdot,\cdot\rrangle\colon V/G\times V/G\to\mathbb{R}$ is defined by
\[
\llangle [x],[z]\rrangle
:=\sup_{\substack{p\in[x]\\q\in[z]}}\langle p,q\rangle.
\]
\item[(b)]
Given \textbf{templates} $z_1,\ldots,z_n\in V$, the corresponding \textbf{max filter bank} $\Phi\colon V/G\to\mathbb{R}^n$ is defined by
\[
\Phi([x])
:=\{\llangle [x],[z_i]\rrangle\}_{i=1}^n.
\]
\end{itemize}
\end{definition}
(Since $G\leq \Oname(V)$, $\llangle [x],[z]\rrangle = \sup_{q\in [z]}\langle x , q\rangle$ is a supremum of linear functionals.) In broad terms, an individual max filtering map $\llangle [\cdot], [z] \rrangle$ is a scalar feature map which takes $[x]\in V/G$ as input and measures the maximal alignment between the orbits $[x]$ and $[z]$ when interpreted as subsets of $V$. Notably, the map is a convex, $\|z\|$-Lipschitz continuous invariant, as it is defined as the supremum of $\|z\|$-Lipschitz linear functionals. A max filter bank consists of a collection of such maps, making it componentwise convex and Lipschitz continuous. Furthermore, it is uniquely determined up to origin-fixing isometries of $V/G$, as demonstrated by the following polarization identity, which holds for all $x,z\in V$:
\begin{equation}
    \label{prop.max filter lemma2 distance}
d([x],[z])^2 = \|x\|^2 - 2\llangle [x],[z]\rrangle + \|z\|^2,
\end{equation}
where we note that $\|x\| = d([x],[0])$.

The max filtering map can also be regarded as a fundamental convex invariant since every convex invariant $f\colon  V\to \mathbb R$ can be expressed as a supremum of affine max filters, that is
\[f(x) = \sup_{z\in \Omega}[\llangle [x], [z]\rrangle + b_z],\]
 for some $\Omega\subseteq  V$ and $\{b_z\}_{z\in\Omega}\in \mathbb R^\Omega$. Additionally, when $G\leq \Oname(V)$ is compact, the subgradient of the max filtering map has an explicit form:
\[\partial\llangle [\cdot], [z]\rrangle|_{x} = \operatorname{conv}\{q\in [z]: \langle x,q\rangle = \llangle[x],[z]\rrangle\}.\]
This follows from the general fact that $\partial(\max_{i\in I}f_i)|_x = \operatorname{conv}\{\nabla f_j(x): j\in \arg\max_{i\in I}f_j(x)\}$, where $\{f_i\}_{i\in I}$ is a collection of convex differentiable functions and $\operatorname{conv}$ denotes the convex hull operator.

From a machine learning perspective, convexity is desirable because max filter banks can serve as $G$-invariant layers in classification models, with their constituent templates as trainable parameters. Relevant numerical examples are discussed in Section~6 of~\cite{CahillIMP:22}. We also hypothesize that one could generalize input convex neural networks \cite{AmosXK:17} to the convex invariant setting using max filters, but we leave this as a direction for future exploration.

A theoretical question posed in~\cite{CahillIMP:22} asks whether every injective max filter bank is bilipschitz. When $G$ is finite, this question was resolved by~\cite{BalanT:23}, which showed that every injective max filter bank admits bilipschitz bounds. However, the question remains open for infinite $G$, with only three exceptions:
\begin{itemize}
    \item Complex phase retrieval~\cite{Alharbi:22,CahillCD:16}, where $V=\mathbb C^d$ and $G=\{z\cdot \operatorname{id}:z\in \mathbb C, |z|=1\}$. 
    \item Polar actions~\cite{MixonQ:22}, where $V/G$ is isometrically isomorphic to $V'/G'$ for some finite $G'\leq \Oname(V')$.
\end{itemize}
 
For general infinite $G$, the bilipschitz property of injective max filter banks remains an open question. While we do not fully resolve this issue, we investigate conditions under which these maps are bilipschitz, given sufficiently many generic templates. Specifically, we prove that this property holds locally near orbits of maximal dimension and hence globally for groups where all nonzero orbits have constant dimension.

\subsection{Main results and paper outline}
Our sufficient bound on the number of generic templates is determined by the following complexity parameter:
\begin{definition}\label{def.regular}
    Let $G\leq \Oname(d)$ be a compact group. The set of \textbf{regular points} is defined as
    \[R(G) := \big\{x\in\mathbb R^d: \dim([x]) = \max_{y\in\mathbb R^d}\dim([y])\big\}.\]
    The \textbf{regular Voronoi complexity} of $G$ is then given by
    \[\chi(G) := \max_{x,p\in R(G)}\{|G_x/G_p|:G_p\leq G_x\},\]
    where $G_y:=\{g\in G:gy=y\}$ denotes the \textbf{stabilizer} of $y$ in $G$.
\end{definition}

Note that the set $R(G)$ is an open, dense, $G$-invariant subset of $\mathbb R^d$ (for example, see Theorems~3.49~and~3.82 in~\cite{AlexandrinoB:15}). Intuitively, $\chi(G)$ measures the maximum relative discrepancy in discrete degrees of freedom among orbits with the highest infinitesimal degrees of freedom, i.e., orbits of regular points.

We now present our two main results. The first establishes sufficient conditions under which max filter banks are locally bilipschitz over $R(G)$, without imposing restrictions on $G$ beyond compactness and orthogonality.
\begin{theorem}
    \label{thm.regular lower lip generic}
    Let $G\leq \operatorname{O}(d)$ be a compact group and define $c:= d- \max_{x\in \mathbb R^d}\dim([x])$. For generic $z_1,\dots, z_n\in \mathbb R^d$, the max filter bank $\Phi\colon \mathbb R^d/G\to \mathbb R^n$ given by $\Phi([x]) := \{\llangle [x],[z_i] \rrangle\}_{i=1}^n$ is locally bilipschitz at every $x\in R(G)$, provided $n > 2\cdot \chi(G)\cdot (c - 1)$.
\end{theorem}
Here, a map $f\colon \mathbb R^d/G \to \mathbb R^n$ is said to be locally bilipschitz at $x$ if it is bilipschitz when restricted to a neighborhood of $x$, with respect to the quotient distance induced by $G$. The above result is core to ensuring that max filter banks can form componentwise convex bilipschitz embeddings when all nonzero vectors lie within $R(G)$. This is the content of the following primary result:
\begin{theorem}
    \label{thm.almost free actions}
    Let $G\leq \operatorname{O}(d)$ be a compact group such that $\mathbb R^d-\{0\}\subseteq R(G)$, and define $c:= d- \max_{x\in \mathbb R^d}\dim([x])$. For generic $z_1,\dots, z_n\in \mathbb R^d$, the max filter bank $\Phi\colon \mathbb R^d/G\to \mathbb R^n$ given by $\Phi([x]) := \{\llangle [x],[z_i] \rrangle\}_{i=1}^n$ is bilipschitz, provided $n > 2\cdot \chi(G)\cdot (c-1)$.
\end{theorem}

In \cref{sec.relevance of results}, we discuss the significance of \cref{thm.almost free actions} in the contexts of stable weighted phase retrieval and a nearest neighbor problem in cryo-EM. In \cref{sec.voronoi BIG SECTION}, we introduce a Voronoi cell decomposition (\cref{def.normal voronoi}) that establishes a geometric framework for the proofs presented in this paper. We also give a key geometric characterization of $R(G)$ (\cref{lem.regular char}) leveraging a desingularization theorem (\cref{lem.desingularization}) which may be of independent interest. The proofs of our main results (\cref{thm.regular lower lip generic,thm.almost free actions}) are presented in \cref{sec.local lower stability}. Finally, we conclude with a discussion in \cref{sec.discussion}.

\subsection{Relevance of results}
\label{sec.relevance of results}
In this section, we highlight the relevance of \cref{thm.almost free actions} in the contexts of \textit{stable weighted phase retrieval} and a nearest neighbor problem in cryogenic electron microscopy (cryo-EM). Throughout, $\mathbb C^*\cong \operatorname{SO}(2) \cong S^1$ (resp.\ $\mathbb H^*\cong \operatorname{SU}(2)\cong S^3$) denotes the group of unit complex numbers (resp.\ unit quaternions).
\subsubsection{Stable weighted phase retrieval}
In this section, let $V$ be a finite dimensional real Hilbert space and let $G\leq \Oname(V)$ be a compact-connected group. The main result of this work (\cref{thm.almost free actions}) establishes that max filter banks are bilipschitz if they include sufficiently many generic templates and if the nonzero orbits of $G$ have constant dimension.

In particular, if $\dim([p])=\dim(G)$ for some $p\in V$, the latter condition corresponds to the action of $G$ being \textit{almost free} on the unit sphere $\mathbb S(V)$, meaning that every nonzero orbit has a finite stabilizer. As noted in Section~3.2 in~\cite{GordoskiL:16} and since $G$ is connected, this happens if and only if $G$ is the image of a representation $\phi\colon K^* \to \Oname(V)$, where $K = \mathbb C$ or $K = \mathbb H$, and $\phi$ is a direct sum of nontrivial irreducible complex (resp.\ quaternionic) representations of $K^*$. In this case, we call $\mathbb S(V)/G$ a \textit{weighted complex (resp.\ quaternionic) projective space}.

In these settings, templates $z_1,\dots, z_n\in V$ are said to achieve \textbf{stable weighted phase retrieval} if the max filter bank $\Phi\colon V/G\to \mathbb R^n$, defined by $\Phi([x])= \{\llangle [x],[z_i]\rrangle\}_{i=1}^n$, is bilipschitz. This leads to the following corollary of \cref{thm.almost free actions}:

\begin{corollary}
    \label{cor.phase}
    Suppose that $G\leq \Oname(V)$ is compact, connected and acting almost freely on $\mathbb S(V)$. Then, $\chi(G)=\max_{x\neq 0}|G_x|$ and  generic templates $z_1,\dots, z_n\in V$ achieve stable weighted phase retrieval provided $n > 2\cdot \chi(G)\cdot (c-1)$, where $c:=\dim_{\mathbb R}(V)-\dim(G)$.
\end{corollary}

This framework generalizes \textit{stable complex phase retrieval} \cite{Alharbi:22,CahillCD:16}, where one seeks to \textit{stably} recover $x\in \mathbb C^d$, up to the equivalence $x\sim e^{\mathrm i\theta} x$, from magnitude measurements $|x^*z_i|= \llangle [x],[z_i]\rrangle$. This setup corresponds to a bilipschitz max filter bank.

In the complex setting, stable weighted phase retrieval generalizes this by incorporating weights into the orbit equivalence relation induced by $G \leq \Oname(\mathbb C^d)$:
\[x \sim \operatorname{diag}\{e^{\mathrm ik_j\theta}\}_{j=1}^d \cdot x, \quad \forall \, \theta \in [-\pi,\pi],\]
where $\{k_j\}_{j=1}^d$ is a collection of nonzero integers, namely the \textit{weights}. Viewing this problem from the lens of $\phi\colon K^*\to O(V)$ as above, we obtain the natural extension to the quaternionic setting.

\subsubsection{Approximate nearest neighbor problem in cryo-EM}

Given a 3D macromolecular complex $P$ with unknown structure, cryo-EM produces a set of 2D noisy projection images $\mathcal I = \{I_j\}_{j=1}^M$, corresponding to different 3D viewing angles. As argued in \cite{ZhaoS:14,ZhaoSS:16}, each image $I_j$ can be approximated as an $L\times L$ pixel sampling of a function $f_j \in L^1(\mathbb R^2)\cap L^2(\mathbb R^2)$, with Fourier transform $\mathcal Ff_j \in L^2(\mathbb R^2)$ compactly supported in the disk $\Omega_{c_0}\subseteq \mathbb R^2$ of radius $c_0$. Furthermore, $\mathcal F f_j$ is approximated in polar coordinates by a finite expansion
    \begin{equation}
        \label{eq.finite expansion}
        \mathcal Ff_j(r,\theta) \approx \sum_{k=-k_{\max}}^{k_{\max}}\sum_{q=1}^{p_k}a_{k,q}^{f_j} \psi_{k,q}^{c_0}(r)e^{ik\theta},
    \end{equation}
    where $\{\psi_{k,q}^{c_0}(r)e^{ik\theta}\}_{k\in \mathbb Z, q\in \mathbb N}$ is the scaled Fourier-Bessel $L^2(\Omega_{c_0})$-orthonormal basis, and one may take $k_{max} = O(L)$ and $p := \sum_{-k_{\max}}^{k_{\max}} p_k = O(L^2)$.
    
    Denoising an image $I_j$ involves finding its nearest neighbors under the \textit{rotational alignment} distance:
    \[d_{\operatorname{SO}(2)}(f,g) := \inf_{R\in \operatorname{SO}(2)}\big\|f - g\circ R^{-1}\big\|_{2}.\]
    Since this is computationally expensive, an approximation via \eqref{eq.finite expansion} is used:
    \[
    d_{\operatorname{SO}(2)}(I_i,I_j) \approx \min_{\alpha\in [0,2\pi)}\big\|\big\{a_{k,q}^{f_i}\big\} - \big\{a_{k,q}^{f_j}e^{\mathrm ik\alpha}\big\}\big\|_{2}.
    \]
    This shifts the setting of the nearest neighbor task into the orbit space of a finite dimensional orthogonal representation $W$ of $\mathbb C^*$ given by the space of coefficients
    \[W:= \big\{\{a_{k,q}\}: |k|\leq k_{\max}, p_k > 0, q\in \{1,\dots,p_k\}\big\},\]
    equipped with the unitary action $z \cdot \{a_{k,q}\} := \{z^k\cdot a_{k,q}\}$, for $z\in\mathbb C^*$.
    By virtue of \cref{ex.bilip nearest}, we seek a \textit{fast} bilipschitz embedding of $W/\mathbb C^*$ into Euclidean space. In \cite{ZhaoS:14}, the authors use the \textit{bispectrum} embedding $\Phi_B\colon W/\mathbb C^* \to \mathbb R^{N}$, with $N=O(L^5)$, defined by
    \[\Phi_B(\mathbb C^*\cdot \{a_{k,q}\}) := \{a_{k_1,q_1}a_{k_2,q_2}\overline{a_{k_1+k_2,q_3}}: p_{k_1},p_{k_2},p_{k_1+k_2} > 0\}.\]
    However, by \cref{prop.nondiff} in \cref{app.nondiff}, $\Phi_B$ fails to be bilipschitz if $p_{k}, p_{k'}>0$ for some $k\neq \pm k'$. The authors in \cite{ZhaoS:14} also suggest precomposing $\Phi_B$ with the continuous scaling $\{a_{k,q}\}\mapsto \left\{{a_{k,q}}/{|a_{k,q}|^{2/3}}\right\}$, but this fails to be upper Lipschitz.

    We propose a replacement by using max filter banks. By projecting away from the trivial component  as in \cref{rem.trivial} below, the problem reduces to stable weighted phase retrieval. Then \cref{cor.phase} applies with $\chi(G) \leq k_{\max} = O(L)$ and $c \leq p = O(L^2)$, and it entails that $O(L^3)$ generic templates constitute bilipschitz max filter banks (cf. $N=O(L^5)$ above.) Moreover, each constituent max filtering map can be approximated by a linear search over samples obtained via the fast Fourier transform. We hypothesize that max filter banks will outperform the bispectrum in accuracy and efficiency, but since our focus is theoretical, we leave the numerical validation of this for future work.

\begin{remark}
    \label{rem.trivial}
    Suppose $G\leq \Oname(W)$ is compact and let $F:= \{w\in W: G_w=G\}$ denote its fixed subspace. Put $V:= F^\perp$ and let $P_F$ and $P_V$ denote corresponding linear orthogonal projections. By the proof of Lemma~39 in~\cite{CahillIM:24} and given a bilipschitz embedding $f\colon V/G \to \mathbb R^n$ with bounds $\alpha \leq \beta$, the map $\Psi\colon W/G \to F\times \mathbb R^n$ defined by $\Psi([x]) := \big(\alpha \cdot P_Fx, \beta \cdot f([P_Vx])\big)$ also has bilipschitz bounds $\alpha\leq \beta$.
\end{remark}

\section{The Voronoi Decomposition}
\label{sec.voronoi BIG SECTION}
In this section, we conduct a geometric analysis of isometric linear actions by introducing a Voronoi cell decomposition of the space. A key result, \cref{lem.regular char}, will play a central role in proving the key \cref{lem.regular local avoidance} of \cref{sec.local lower stability}.

We begin in \cref{sec.voronoi intro} by introducing the decomposition and presenting the main results of this section: \cref{lem.PG char,lem.regular char}. In \cref{sec.concrete example}, we provide a concrete example to illustrate the underlying concepts. \cref{sec.voronoi prelim} covers the necessary preliminaries, while \cref{sec.voronoi properties} establishes and proves key properties of the Voronoi decomposition. Finally, \cref{sec.PG char proof,sec.RG char proof} contain the proofs of \cref{lem.PG char,lem.regular char}, respectively.
\subsection{Setup and main results}
\label{sec.voronoi intro}
For a subset $S\subseteq \mathbb R^d$, let $\operatorname{relint}(S)$ denote the interior of $S$ relative to its affine span $\operatorname{aff}(S)$.
\begin{definition}
    \label{def.normal voronoi}
    Let $G\leq \operatorname{O}(d)$ be a compact group. For $x\in \mathbb R^d$, the \textbf{unique Voronoi cell} of $x$, denoted $U_x$, is defined by:
    \[z \in U_x \iff \{x\} = \arg\max_{p\in [x]}\langle p,z\rangle \iff \{x\} = \arg\min_{p\in [x]}\| z - p\|.\]
    The \textbf{open Voronoi cell} of $x$ is then defined as:
    \[V_{x} := \operatorname{relint}\big(U_{x}\big),\]
    and the \textbf{open Voronoi diagram} of $x$ is given by
    \[Q_{x} := \bigsqcup_{p\in [x]}V_{p}.\]
\end{definition}

The equivalence of the two characterizations of $U_x$ follows directly from the polarization identity $\|z-p\|^2 = \|x\|^2 + \|z\|^2 - 2\langle p,z\rangle$ for $p\in [x]$. In \cref{sec.voronoi properties}, we explore the properties of this decomposition and provide characterizations for its components.

The distinction between $U_x$ and its relative interior $V_x$ is important. First, the definition is non-redundant, as $U_x\neq V_x$ may occur (e.g., see \cref{ex.circle UXVX}.) Second, $V_x$ is functionally essential, as demonstrated in the proof of \cref{lem.regular local avoidance} in \cref{sec.local lower stability}. There, we use \cref{lem.regular char}, which provides a geometric “interchangeability” characterization of $V_x$ when $[x]$ has maximal dimension, i.e., $x\in R(G)$. To state both of our Voronoi `interchangeability' results, we first introduce the concept of principality.

\begin{definition}\label{def.principal}
    Let $G\leq \Oname(d)$ be compact. The set of \textbf{principal points} is defined as
    \[P(G) := \big\{x\in\mathbb R^d: \forall\, z\in\mathbb R^d, G_z \leq G_x \implies G_z = G_x\}.\]
\end{definition}

Notably, $P(G)\subseteq R(G)$ is a $G$-invariant open and dense subset of $\mathbb R^d$, as established in Theorems~3.49~and~3.82 in~\cite{AlexandrinoB:15}. Intuitively, orbits of principal points possess maximal degrees of freedom, both in the infinitesemal and discrete sense.

In \cref{sec.PG char proof}, we prove the following geometric characterization of principal points in terms of open Voronoi cells.
\begin{theorem}
    \label{lem.PG char}
    Let $G\leq \operatorname{O}(d)$ be compact. The following are equivalent:
    \begin{itemize}
        \item[(a)] $x\in P(G)$.
        \item[(b)] $z\in V_{x}$ implies $x\in V_z$.
    \end{itemize}
\end{theorem}
\begin{proof}
    See \cref{sec.PG char proof}.
\end{proof}
Building on this, we refine the Voronoi decomposition to the local level, enabling a similar characterization for regular points.
\begin{definition}
    Let $G\leq \Oname(d)$ be compact. For $z\in\mathbb R^d$, the \textbf{local open Voronoi cell} $V_{z}^{loc}$, is defined as follows: $x\in V_{z}^{loc}$ if there exist open neighborhoods $U$ of $z$ and $V$ of $x$ such that:
    \[\forall q\in V, \ \ \Big|\arg\sup_{p\in[z]\cap U}\langle p,q\rangle \Big|= 1.\]
\end{definition}
\begin{samepage}
\begin{theorem}\label{lem.regular char}
    Let $G\leq \Oname(d)$ be compact. For $x\in \mathbb R^d$, the following are equivalent:
    \begin{lemenum}
        \item $x\in R(G)$.
        \item $z\in V_x$ implies $x\in V_{z}^{loc}$.
    \end{lemenum}
\end{theorem}
\begin{proof}
See \cref{sec.RG char proof}.
\end{proof}
\end{samepage}
Although \cref{lem.PG char} is not directly referenced later in the paper, it is an independently interesting geometric result. Additionally, its proof in \cref{sec.PG char proof} provides preparatory groundwork for the more technical and analogous proof of \cref{lem.regular char} in \cref{sec.RG char proof}.\subsection{Concrete example}
\label{sec.concrete example}
In this section, we provide a concrete three-dimensional example to illustrate the definitions and theorems we have presented in this section so far. In this example, we will observe that the action behaves `nicely’ in the sense that $U_x = V_x$ for all $x \in \mathbb{R}^3$. However, as we will see later in \cref{ex.circle UXVX}, this is not always the case.

\begin{figure}[t]
    \centering
    \includegraphics[scale=0.8]{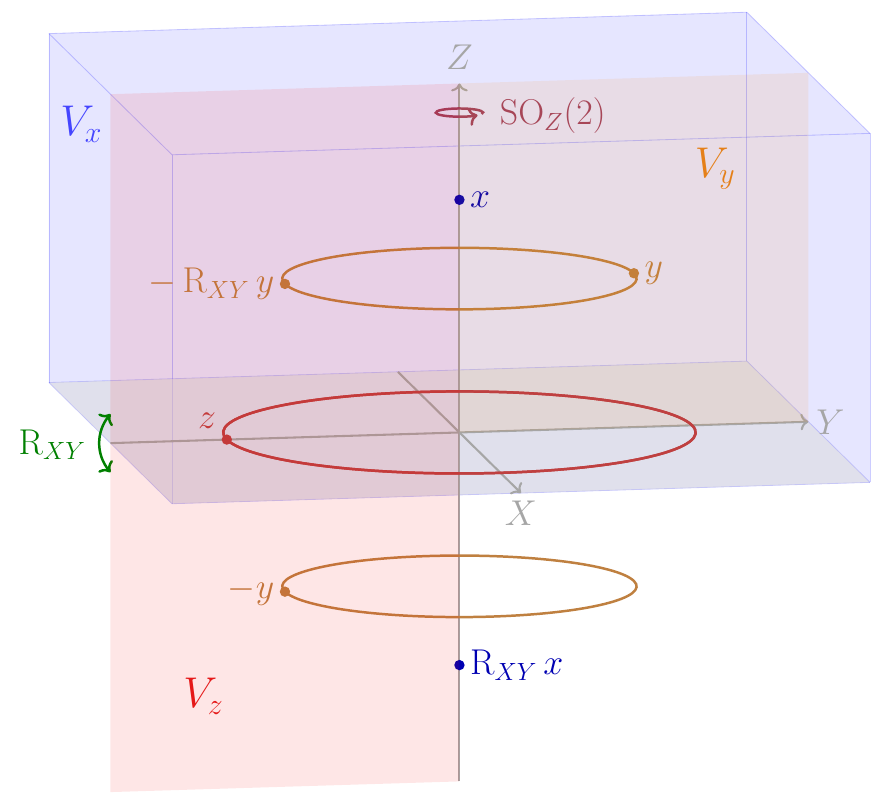}
    \caption{Illustration for \cref{ex.3d voronoi SOReflection}, showing instances of $x$, $y$, and $z$ along with their orbits and Voronoi cells, all of which are precisely described in the referenced example.}
    \label{fig:3d voronoi SOReflection}
  \end{figure}

\begin{example}\label{ex.3d voronoi SOReflection}
Suppose $G\leq \Oname(3)$ is the commutative group generated by $\operatorname{SO}_Z(2)$, the subgroup consisting of all counterclockwise rotations around the $Z$-axis, and $R_{XY}$, the reflection across the $XY$-plane. We begin by computing $U_w$ and $V_w^{loc}$ for each $w\in\mathbb R^3$, dividing analysis into cases; see \cref{fig:3d voronoi SOReflection}.

For $x\in \{(0,0,x_3)\in\mathbb R^3: x_3>0\}$, it holds that $V_x^{loc} = \mathbb R^3$ since all points project uniquely onto $\{x\}$. Next, $U_x$ is the open upper half-space (depicted in blue in~\cref{fig:3d voronoi SOReflection}.) Similarly, $U_{R_{XY}x}=R_{XY}U_{x}$. This covers the case of nonregular nonzero points.

For $y \in \{(0,x_2,x_3)\in\mathbb R^3: x_2,x_3 > 0\}$, it holds that $V_y^{loc}$ is the complement of the $Z$-axis, as one can take a tubular neighborhood $U$ of $\operatorname{SO}_Z(2)\cdot y$ so that all points, except those lying on the $Z$-axis, project uniquely onto $[y]\cap U=\operatorname{SO}_Z(2)\cdot y$. Next, $U_y$ is the open first quadrant of the $YZ$-plane (shown in orange in~\cref{fig:3d voronoi SOReflection}.) Additionally, $U_{gy}=gU_y$ for all $g\in G$. This covers the case of principal points.

For $z\in \{(0,-x_2,0)\in\mathbb R^3: x_2 > 0\}$, we have that $V_{z}^{loc}$ is again the complement of the $Z$-axis by a similar argument as in the previous paragraph. Next, $U_z$ is given by the open half-plane $\{(0,-x_2,x_3)\in\mathbb R^3: x_2 > 0, x_3\in\mathbb R\}$. Additionally, $U_{R_{XY}z} = U_z$ and $U_{gz} = gU_z$ for each $g\in G$. This covers the case of regular nonprincipal points.

Lastly, $U_{(0,0,0)} = \mathbb R^3$ and $V_{(0,0,0)}^{loc} = \mathbb R^3$.

Note that $V_w=U_w$ for each $w\in \mathbb R^3$ since in each of the above cases, $U_w$ is open in its affine hull. Moreover, $Q_x$ is the complement of the $XY$-plane, $Q_z$ is the complement of the $Z$-axis, $Q_y = Q_x\cap Q_z$, and $Q_{(0,0,0)} = \mathbb R^3$.

Finally, we verify the statements of \cref{lem.PG char,lem.regular char}. For each $q\in V_y$, it holds that $y\in V_q=V_y$. Since $y\in P(G)$, this is consistent with the implication (a)$\Rightarrow$(b) in \cref{lem.PG char}. On the other hand, observe that $-y\in V_z$, $z\notin V_{-y}$, and $z\in V_{-y}^{loc}$. Since $z\notin P(G)$, this is consistent with the implication (b)$\Rightarrow$(a) in \cref{lem.PG char}. Moreover, since $z\in R(G)$, this is consistent with the implication (a)$\Rightarrow$(b) in \cref{lem.regular char}. Lastly, the statements that $x\notin R(G)$, $y\in V_x$ and $x\notin V_y^{loc}$ are consistent with the implication (b)$\Rightarrow$(a) in \cref{lem.regular char}.
\end{example}

\subsection{Preliminary Results}
\label{sec.voronoi prelim}
We begin with a crucial preliminary result drawn from the theory of nonlinear orthogonal projection on manifolds, which will be referenced frequently throughout the paper. A visual illustration is provided in~\cref{fig:manifold nonlinear orthogonal}. For $x\neq y\in \mathbb R^d$, let $(x,y]$ denote the line segment from $x$ to $y$, which includes $y$ but excludes $x$, and let $[x,y]$ denote the segment that includes both $x$ and $y$; additionally, define $(x,x] := [x,x] = \{x\}$. 

\begin{proposition}[Remark~3.1, Corollary~3.9, Theorem~3.13a and Theorem~4.1 in~\cite{DudekH:94}]
    \label{prop.daduk proposition}
    Let $M$ be a smooth embedded submanifold of $\mathbb R^{d}$. For $z\in\mathbb R^d$ and $x\in \arg\min_{p\in M}\| z-p\|$, each of the following statements holds:
    \begin{propenum}
        \item \label{prop.daduk argmax in normal space} $z\in N_xM$, the orthogonal complement of the tangent space to $M$ at $x$.
        \item \label{prop.daduk interval in Omega} There exists an open neighborhood $U$ of $(z,x]$ such that $\{x\} = \arg\min_{p\in M}\| n-p\|$ for $n\in U\cap N_xM$ and $\big|\arg\min_{p\in M}\| t-p\|\big|= 1$ for all $t\in U$. Moreover, the map $v_x$, which sends $t\in U$ to the unique element in $\arg\min_{p\in M}\| t-p\|$, is smooth over $U$.
        \item \label{prop.daduk Omega normal neighborhood} If there exists an open neighborhood $W_z$ around $z$ such that $|\arg\min_{p\in M}\| t-p\||=1$ for all $t\in W_z$, then the neighborhood $U$ in (b) can be enlarged to include $z$.
    \end{propenum} 
\end{proposition}    
\begin{remark} If $M\subseteq \mathbb R^d$ lies on a sphere centered at the origin (e.g., $M = G\cdot x$ for some $G\leq \operatorname{O}(d)$), then for all $z\in \mathbb R^d$, the polarization identity gives
\[\arg\min_{p\in M}\| z-p\| = \arg\max_{p\in M}\langle p,z\rangle.\]
\end{remark}

\begin{figure}[t]
    \centering
    \includegraphics[scale=1.2]{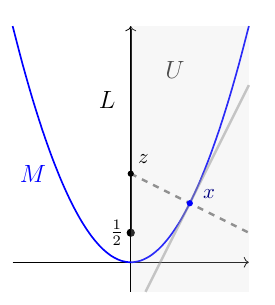}
    \includegraphics[scale=1.2]{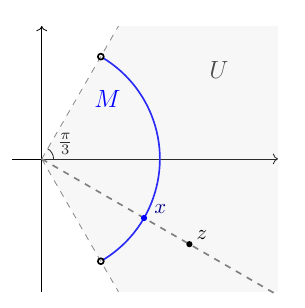}
    \caption{Illustration for \cref{prop.daduk proposition}. \textbf{(left)} Here, $M = \{(t,t^2)\in\mathbb R^2: t\in\mathbb R\}$ and $U = \{(x_1,x_2)\in\mathbb R^2:x_1 > 0\}$. The line $L = \{(0,x_2)\in\mathbb R^2:x_2 > \frac{1}{2}\}$ is the set of points with multiple nearest neighbors to $M$, and its closure is $\overline L = L\cup \{(0,\frac{1}{2})\}$. For each $z\in U \cup L$ and $x \in U\cap \arg\min_{p\in M}\|p-z\|$, the neighborhood $U$ satisfies assertions (b) and (c) in \cref{prop.daduk proposition}. For $z_0\in R:= \{(0,x_2)\in\mathbb R^2:x_2 < \frac{1}{2}\}$, it holds that $\arg\min_{p\in M}\|p-z\| = \{(0,0)\}$, and any neighborhood $U$ of $R$ with $U\cap \overline L = \varnothing$ satisfies those assertions with respect to each $z_0\in R$. \textbf{(right)} Here, $M = \{(\cos(\theta),\sin(\theta))\in\mathbb R^2: \theta\in (-\frac{\pi}{3},\frac{\pi}{3})\}$ and $U = \operatorname{int}(\operatorname{cone}(M))$. For each $z\in U\cup \{(0,0)\}$ and $x \in \arg\min_{p\in M}\|p-z\|$, the neighborhood $U$ satisfies assertions (b) and (c) in \cref{prop.daduk proposition}. For nonzero $z\in U^c$, the set $\arg\min_{p\in M}\|p-z\|$ is empty.}
    \label{fig:manifold nonlinear orthogonal}
  \end{figure}

Next, for a compact group $G\leq \operatorname{O}(d)$ with Lie algebra $\mathfrak g$ and $x\in \mathbb R^d$, the orbit $G\cdot x$ is an embedded submanifold of $\mathbb R^d$ (Proposition~3.41 in~\cite{AlexandrinoB:15}). We define the \textbf{tangent space} at $x$ to its orbit $G\cdot x$ by
\[T_x := \mathfrak g \cdot x,\] 
and the \textbf{normal space} at $x$ to its orbit $G \cdot x$ as:
\[N_x:=(\mathfrak g\cdot x)^{\perp}.\]
We have the following (presumably folklore) result which establishes the $G_x$-invariance of $T_x$ and $N_x$. While we could not locate a reference, we provide a proof.
\begin{proposition}
    \label{prop.Nx Tx Gx invariant}
   Let $G\leq \Oname(d)$ be a compact group and fix $x\in\mathbb R^d$. Then, for each $h\in G$, we have $T_{hx}=h\cdot T_x$ and $N_{hx} = h\cdot N_x$. In particular, $T_x\oplus N_x$ is a $G_x$-invariant orthogonal decomposition of $\mathbb R^d$.
\end{proposition}
\begin{proof}
    Since each $h\in G$ is an isometry, it suffices to show that $T_{hx}=h\cdot T_x$; this would then imply that $h\cdot N_x = h\cdot(T_x)^\perp = (T_{hx})^{\perp} = N_{hx}$. Since $h$ is linear and $\dim(T_x) = \dim(T_{hx})$,  it suffices to show that $h\cdot T_x \subseteq T_{hx}$ for each $h\in G$. To this end, fix $h\in G$ and $t\in T_x$. We aim to show that $ht\in T_{hx}$. By the definition of $T_x$, there exists $\omega\in \mathfrak g$ such that $t = \omega \cdot x$. Since $\mathfrak g$ is the tangent space to $G$ at its identity $e$, there exists $\varepsilon >0$ and a smooth curve $\alpha\colon (-\varepsilon,\varepsilon)\to G$ such that $\alpha(0) = e$ and $\alpha'(0)=\omega$. The smoothness of the action of $G$ on $\mathbb R^d$ implies:
    \[t = \omega\cdot x = \alpha'(0)\cdot x = \left.\frac{d(\alpha(t) x)}{dt}\right|_{t=0}.\]
    Now, let $\gamma := h\alpha h^{-1}\colon (-\varepsilon,\varepsilon)\to G$ denote the conjugation of $\alpha$ by $h$. Then $\gamma$ is a smooth curve satisfying $\gamma(0) = e$ and $\gamma'(0) \in \mathfrak g$. Thus
    \[ht = h\cdot\left.\frac{d(\alpha(t) x)}{dt}\right|_{t=0} = \left.\frac{d(h\alpha(t)h^{-1}  (hx))}{dt}\right|_{t=0} = \gamma'(0)\cdot hx \in T_{hx}.\]
\end{proof}

\subsection{Properties of the Voronoi Decomposition}
\label{sec.voronoi properties}
The following lemma explores the properties of $U_x$ and $V_x$, and the third statement provides justification for using a disjoint union in the definition of $Q_x$. 

\begin{lemma}
    \label{lem.UV facts}
    Let $G\leq \operatorname{O}(d)$ be a compact group. For $x\in \mathbb R^d$, each of the following statements holds:

    \begin{lemenum}
        \item \label{lem.UV facts Gx char}$z\in U_x$ if and only if $G_x = \{g\in G: \llangle [z],[x]\rrangle = \langle gz, x\rangle\}$.
        \item \label{lem.UV facts equivariance}$U_{gx} = g\cdot U_{x}$ and $V_{gx} = g\cdot V_{x}$ for all $g\in G$.
        \item \label{lem.UV facts intersection}For $q_1,q_2\in [x]$, if $U_{q_1}\cap \overline{U_{q_2}} \neq \varnothing$,  then $q_1=q_2$.
        \item \label{lem.UV facts intersection V open in normal}$\operatorname{aff}(U_x) = \operatorname{span}(U_x) = N_x$, and $V_x$ is a star convex open neighborhood of $x$ in $N_x$.
        \item \label{lem.V char} The following characterization holds:
        \[z\in V_x \Longleftrightarrow z\in U_x \wedge |\arg\max_{p\in[x]}\langle p,t\rangle|=1 \text{ for $t$ in a neighborhood of $z$}.\]
        \item \label{lem.UV facts semialgebraic}The sets $N_x$, $U_x$ and $V_x$ are semialgebraic subsets of $\mathbb R^d$.
    \end{lemenum}

\end{lemma}
\begin{proof}
    The proofs of (a), (b) and (c) are straightforward.

    To prove (d), note that $U_x \subseteq N_x$ by \cref{prop.daduk argmax in normal space}, so we have $\operatorname{aff}(U_x) \subseteq \operatorname{span}(U_x) \subseteq N_x$. Since $x\in U_x$, \cref{prop.daduk interval in Omega} guarantees the existence of an open neighborhood $U$ of $(x,x] = \{x\}$ in $\mathbb R^d$ such that $U\cap N_x$ is a nonempty subset of $U_x$. This implies that $N_x = \operatorname{aff}(U\cap N_x) \subseteq \operatorname{aff}(U_x)$, so $\operatorname{aff}(U_x) = \operatorname{span}(U_x) = N_x$. Furthermore, $U$ is witnesses that $V_x$ is open in $N_x$ and that $x\in V_x$. For the star convexity of $V_x$ at $x$, let $z\in V_x$ and note that by \cref{prop.daduk interval in Omega}, there exists an open neighborhood $U'$ of $(z,x]$ such that $(z,x] \subseteq U'\cap N_x \subseteq V_x$.

    For the forward implication in (e), the openness of $V_x$ in $N_x$ and its star convexity at $x$ imply that there exists $q\in V_x$ such that $[z,x]\subseteq (q,x]\subseteq V_x$. Then the desired result follows from \cref{prop.daduk interval in Omega} applied to the interval $(q,x]$.
    Next, the reverse implication in (e) follows immediately from \cref{prop.daduk interval in Omega,prop.daduk Omega normal neighborhood} applied to the interval $[z,x]$.

    Lastly, (f) follows from a straightforward argument in first-order logic. A restatement and proof can be found in \cref{prop.distance is semialgebraic}.
\end{proof}

As a first application of \cref{lem.UV facts}, we derive stabilizer inclusions implied by the Voronoi decomposition.
\begin{lemma}
    \label{prop.stab subset lemma whole}
Let $G\leq \Oname(d)$ be a compact group. For $x,z\in\mathbb R^d$, each of the following statements holds:
\begin{lemenum}
    \item \label{prop.stab subset}$z\in U_x$ implies $G_z \leq G_x$.
    \item \label{prop.stab PG subset}$x \in P(G)$ and $z\in N_x$ imply $G_x \leq G_z$.
    \item \label{prop.stab PG to PG}$x\in P(G)$ and $z\in U_x$ imply $G_x = G_z$ and $z\in P(G)$.
\end{lemenum}
\end{lemma}
Intuitively, (a) states that if $z$ uniquely projects to $x$ within $[x]$, then $z$ has `$G_x/G_z$' more degrees of freedom than $x$. (b) states that if $x\in P(G)$, meaning $x$ has maximal degrees of freedom among all orbits, then $N_x$ is fixed by $G_x$, i.e., while fixing $x$, $G_x$ does not introduce any degrees of freedom to $N_x$. Finally, (c) combines (a) and (b).
\begin{proof}[Proof of \cref{prop.stab subset lemma whole}]
    First, we address (a). For $z\in U_x$ and $g\in G_z$, we have $z=gz\in U_x$. By \cref{lem.UV facts equivariance}, we have $z \in U_{g^{-1}x}\cap U_x$, and by \cref{lem.UV facts intersection}, it follows that $g\in G_x$, as required.

    Next, we prove (b). For each $y\in U_x$, (a) gives that $G_y\leq G_x$. Since $x\in P(G)$, it follows that $G_y=G_x$. Thus, $G_x$ fixes $U_x$, and by linearity and \cref{lem.UV facts intersection V open in normal}, we conclude that $G_x$ fixes $\operatorname{span}(U_x) = N_x$.

    Finally, (c) follows immediately from (a) and (b).
\end{proof}

The following lemma elaborates on the properties of the open Voronoi diagram $Q_x$.

\begin{lemma}
    \label{lem.Q lemma}
    Let $G\leq \operatorname{O}(d)$ be a compact group. For $x\in\mathbb R^d$, each of the following statements holds:
    \begin{lemenum}
        \item \label{lem.Q invariant} $Q_{gx} = Q_x = g\cdot Q_x = G\cdot V_x$ for all $g\in G$.
        \item \label{lem.Q semialgebraic} The set $Q_x$ is a semialgebraic subset of $\mathbb R^d$.
        \item \label{lem. Q char}The following characterization holds:
        \[z\in Q_x \Longleftrightarrow |\arg\max_{p\in[x]}\langle p,t\rangle|=1 \text{ for $t$ in a neighborhood of $z$}.\]
        \item \label{lem.Q open dense} $Q_x$ is an open and dense subset of $\mathbb R^d$.
    \end{lemenum}
\end{lemma}
\begin{proof}
    The proofs of (a) is straightforward. 

    Part (b) follows from a standard argument using first-order logic. A restatement and proof are provided in \cref{prop.distance is semialgebraic}.
    
    The proof of (c) directly follows from \cref{lem.V char}.
    
    Finally, we address (d). The openness of $Q_x$ follows immediately from the openness of the characterization in (c). For denseness, let $z\in \mathbb R^d$ and choose $q\in \arg\max_{p\in[x]}\langle p,z\rangle$, which is possible because $[x]$ is closed. By \cref{prop.daduk interval in Omega}, there exists an open neighborhood $U$ of $(z,q]$ such that $|\arg\max_{p\in[x]}\langle p,u\rangle|=1$ for all $u\in U$. By (c), we have $U\subseteq Q_x$. Since $z\in \overline{(z,q]}\subseteq \overline U\subseteq \overline{Q_x}$, the denseness of $Q_x$ follows.
\end{proof}

As a first application of \cref{lem.Q lemma}, we identify the unitary representations of the circle group for which $U_x=V_x$ for all $x\in\mathbb R^d$.
\begin{example}
    \label{ex.circle UXVX}
    Let $k_1,\dots,k_d\in\mathbb Z$ be fixed integers, and let $G\leq \operatorname U(d)$ be the commutative group defined by $G:=\{\operatorname{diag}(\{e^{\mathrm ik_j\theta}\}_{j=1}^d):\theta \in [-\pi,\pi]\}$. We compute the max filter and then perform a case analysis on $\{k_j\}_{j=1}^d$ to determine when $U_x=V_x$ for all $x\in\mathbb R^d$.

    Suppose, without loss of generality, that $\{k_j\}_{j=1}^d$ is sorted in ascending order. Then there exist $l \in \mathbb N$, $d_1,\dots,d_l\in \mathbb N$, and distinct weights $w_1,\dots, w_l\in \mathbb Z$ such that $k_j = w_m$ for $j \in r_m := \{d_{m-1} + 1, \dots, d_m\}$ (with $d_0 := 0$.) For $x\in \mathbb R^d$, we have an orthogonal decomposition $x = \sum_{m=1}^l x_m$, where $(x_m)_j := 1_{j\in r_m}\cdot (x)_j$ for each $j\in \{1,\dots,d\}$.

    With this notation, the max filter is given by
    \[\llangle [x],[z]\rrangle = \max_{-\pi\leq \theta\leq \pi}\, \operatorname{Re}\left(\sum_{m=1}^l x_m^* z_m e^{\mathrm i w_m \theta}\right).\]

    We claim that $U_x = V_x$ holds for all $x\in\mathbb R^d$ if and only if $l=1$ or $l=2$ with  $0\in \{w_1,w_2, w_1+w_2\}$. These cases correspond to the nontrivial component of the action having a spherical quotient diffeomorphic to a standard complex projective space.

    \textbf{Case 1.} Suppose that $l=2$ with $0\in \{w_1,w_2\}$. This case reduces immediately to $l=1$. For if $w_1=0$, then it holds that $U_x = \mathbb C^{d_1}\times U_{x_2}$ and $V_x = \mathbb C^{d_1}\times V_{x_2}$ where $U_{x_2}$ and $V_{x_2}$ correspond to the case $l=1$ and unique weight $w_2$. 
    
    \textbf{Case 2.} Suppose that $l=2$ with $w_1=-w_2$. This case is orthogonally equivalent the case $l=1$ with unique weight $w_1$. This can be seen by the map $x=x_1+x_2 \mapsto x_1 + \overline{x_2}$, which is an orthogonal (nonunitary) transformation of space.

    \textbf{Case 3.} Suppose that $l=1$ and $x\neq 0$ (indeed, $U_0 = V_0 = \mathbb C^d$.) Then
    \[\llangle [x],[z]\rrangle = \max_{-\pi\leq \theta\leq \pi}\, \operatorname{Re}\left(x^* z e^{\mathrm i w_1\theta}\right) = |x^*z|.\]
    If $x^*z = 0$, then $(gx)^*z = 0$ for each $g\in G$ implying $z\notin G\cdot U_x$. Otherwise, if $x^*z = re^{\mathrm i\phi}$, the condition $\operatorname{Re}(re^{\mathrm i\phi}e^{\mathrm i w_1\theta}) = r$ holds if and only if $\phi+w_1\theta \in 2\pi\mathbb Z$, which uniquely determines $e^{\mathrm i w_1\theta}$. This implies that $z\in G\cdot U_x$. As such
    \[G\cdot U_x = \{z\in \mathbb R^d: x^*z \neq 0\}.\]
    
    Since $x^*z=0$ is a closed condition, it follows that $G\cdot U_x$ is open and hence equal to $Q_x = G\cdot V_x$ by \cref{lem. Q char}. By \cref{lem.UV facts intersection}, we conclude that $U_x = U_x \cap G\cdot U_x = U_x \cap G\cdot V_x = V_x$ for all $x\in \mathbb R^d$.

    \begin{figure}[t]
        \centering
        \includegraphics[scale=1.5]{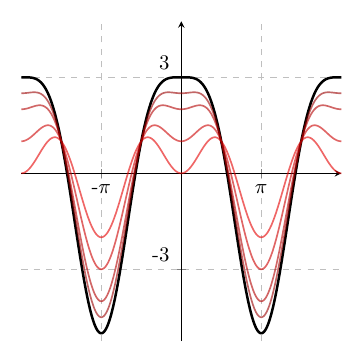}
        \caption{Illustration for Case~4 in~\cref{ex.circle UXVX} when $w_1=1$ and $w_2 = 2$. The graph with the highest $y$-intercept of $3$ corresponds to the function $y = 4\cos(x)-\cos(2x)$. It has a unique global maximum over $[-\pi,\pi]$, attained as a flat local maximum at $x=0$. The other graphs, with $y$-intercepts $k-1$, correspond to the functions $y = k\cos(x)-\cos(2x)$ for $k\in \{1,2,3,3.5\}$. Each attains its global maximum over $[-\pi,\pi]$ at $x=\pm \cos^{-1}\left(k/4\right)\neq 0$. While the values of $k$ here do not exactly correspond to $4-\frac{1}{n}$, the behavior of the global maxima remains the same for that sequence. (We thank Aleksei Kulikov for bringing this example to our attention.)}
        \label{fig:cosine plot}
      \end{figure}    

    \textbf{Case 4.} Suppose that $l\geq 2$ and $w_1, w_2, w_1+w_2\neq 0$. Define $x\in \mathbb R^d$ and a sequence $z_n\in\mathbb R^d$ as follows:
    \[
    \begin{array}{cc}
        (x)_j = \begin{cases}
            1 &\text{if } j\in\{1,d_1+1\},\\
            0 &\text{otherwise,}
            \end{cases}    
        & \text{\quad and \quad}
        (z_n)_j = \begin{cases}
            w_2^2 - \frac{1}{n} &\text{if } j= 1,\\
            -w_1^2 &\text{if } j= d_1+1,\\
            0 &\text{otherwise.}
            \end{cases}    
        \end{array}
        \]
    Then 
    \[\llangle [x],[z_n]\rrangle = \max_{-\pi\leq \theta\leq \pi} \left(w_2^2 - \frac{1}{n}\right)\cos(w_1\theta) -w_1^2\cos(w_2\theta).\]
    Let $z:= \lim_{n\to \infty}z_n$. Then $z\in U_x$, but $z_n\notin Q_x$ for each $n$ (see \cref{fig:cosine plot}.) Since $Q_x$ is open (\cref{lem.Q open dense}), we have $z\in U_x\cap Q_x^c = U_x\cap V_x^c$, which means that $U_x\neq V_x$, as desired.
\end{example}

\subsection{Geometric Characterization of Principal Orbits}
\label{sec.PG char proof}
This section is dedicated to proving \cref{lem.PG char}. To enhance readability and help absorb the core ideas, we recommend skipping the proofs of \cref{prop.principal stratum,lem.PG manifold hypo,lem.Voronoi char minimal geo ext} on the first read. To appreciate the nuance in \cref{lem.PG char}, consider the following easier characterization of principality.
\begin{lemma}\label{lem.PG easier char}
Suppose $G\leq \operatorname O(d)$ is compact. For $x\in \mathbb R^d$, the following statements are equivalent:
    \begin{lemenum}
        \item $x\in P(G)$.
        \item $z\in U_x$ implies $x\in U_{z}$.
    \end{lemenum}
\end{lemma}
\begin{proof}
(a)$\Rightarrow$(b). Let $z\in U_x$. Since $x\in P(G)$ by assumption, \cref{prop.stab PG to PG} implies that $G_z = G_x$. By \cref{lem.UV facts Gx char}, it follows that
\[G_z = G_x^{-1} = \{g^{-1}\in G: \llangle [z],[x]\rrangle = \langle gz, x\rangle\} = \{g\in G: \llangle [x],[z]\rrangle = \langle gx, z\rangle\}.\]
Thus, another application of \cref{lem.UV facts Gx char} shows that $x\in U_z$.

(b)$\Rightarrow$(a). As noted after \cref{def.principal} and by \cref{lem.Q invariant,lem.Q open dense}, we have that $P(G)$ and $Q_x \subseteq G\cdot U_x$ are $G$-invariant, open, and dense in $\mathbb R^d$. As such, we can choose $z\in P(G)\cap U_x$. Then $x\in U_z$ by assumption, and we conclude that $x\in P(G)$ by \cref{prop.stab PG to PG}.
\end{proof}
Intuitively, the proof of the implication (a)$\Rightarrow$(b) in \cref{lem.PG easier char} may be interpreted as follows: if $x$ has maximal degrees of freedom (i.e., $x\in P(G)$) and $z$ projects uniquely to $x$ in $[x]$ (i.e., $z\in U_x$), then $z$ has the same degrees of freedom as $x$ (i.e., $G_x=G_z$) and so $x$ has to project uniquely onto $z$ in $[z]$ (i.e., $x\in U_z$).

Now, \cref{lem.PG char} asks for the stronger interchangeability in relative interiors of unique Voronoi cells. The trickiest implication to prove is (a)$\Rightarrow$(b). We begin by recalling the manifold structure of principal points in the orbit space.

\begin{proposition}
    \label{prop.principal stratum}
    Suppose $G\leq \operatorname O(d)$ is compact. Then, each of the following statements holds
    \begin{itemize}
    \item[(a)]\label{prop.principal stratum A} $[P(G)]$ is an open and dense connected manifold in $\mathbb R^d/G$. It admits a unique Riemannian structure, whose geodesic distance agrees with the quotient distance, and where the restricted orbit map $[\cdot]\big|_{P(G)}\colon P(G)\to [P(G)]$ is a Riemannian submersion.
    \item[(b)]\label{prop.principal stratum B} For each $x,y\in P(G)$, let $C([x],[y])$ denote the set of minimal geodesics joining $[x]$ to $[y]$ in $[P(G)]$. Then, there exists a bijection
    \[
        \arg\min_{q\in[y]}\|q-x\| \longrightarrow C([x],[y])
    \]
    induced by projecting straight lines joining $x$ to $\arg\min_{q\in[y]}\|q-x\|$ into the orbit space.
    \end{itemize}
\end{proposition}
\begin{proof}

    For (a), the openness, denseness and connectedness of $[P(G)]$ follow from Theorem~3.82 in~\cite{AlexandrinoB:15}. The rest of the statement regarding the unique Riemannian structure follows from Exercise~3.81 in~\cite{AlexandrinoB:15}.
     
    For (b), Kleiner's Lemma (Lemma~3.70 in~\cite{AlexandrinoB:15}) and \cref{prop.daduk argmax in normal space} imply that straight lines joining $x$ to $\arg\min_{q\in[y]}\|q-x\|$ lie entirely in $P(G)$ and $N_x$, respectively. In particular, \cref{prop.stab PG subset} implies that every element of $\arg\min_{q\in[y]}\|q-x\|$ is fixed by $G_x$. Then, Lemma~40 in~\cite{MixonQ:22} yields the desired bijection
    \[\arg\min_{q\in[y]}\|q-x\| = \frac{\arg\min_{q\in[y]}\|q-x\|}{G_x} \longrightarrow C([x],[y]),\]
    induced by projecting (unit speed) straight lines joining $x$ to $\arg\min_{q\in[y]}\|q-x\|$ into the orbit space wherein they land in $[P(G)]$.
\end{proof}
Next, we reformulate the implication (a) $\Rightarrow$ (b) in \cref{lem.PG char} using a `cut-locus' interchangeability argument for the Riemannian manifold $[P(G)]$, via the following lemma.
\begin{lemma}\label{lem.Voronoi char minimal geo ext}
Suppose $G\leq \Oname(d)$ is compact. For $x,z \in P(G)$, the following statements are equivalent:
\begin{lemenum}
    \item $z\in Q_x$.
    \item There exists $q\in P(G)$ and a minimal geodesic $\gamma^r$ in $[P(G)]$ joining $[x]$ to $[q]$ such that $[z]$ lies in the interior of $\gamma^r$.
\end{lemenum}
\end{lemma}
\begin{proof}
    (a)$\Rightarrow$(b). Since $z\in Q_x$, there exists $g\in G$ such that $z \in g\cdot V_x = V_{gx}$. By openness of $V_x$ in $N_x$ and its star convexity at $x$ as shown in \cref{lem.UV facts intersection V open in normal}, there exists $q\in V_x$ such that $[x,g^{-1}z] \subseteq [x,q)\subseteq V_x$. By \cref{prop.stab PG to PG}, we have $[x,q]\subseteq V_x\subseteq P(G)$. Then the desired conclusion follows from \cref{prop.principal stratum B} by projecting the (unique) straight line minimizer $[q,x]$ joining $q$ to $\arg\min_{p\in [x]}\|q-p\| = \{x\}$ via the orbit map and reversing the parameterization of the resulting geodesic in $[P(G)]$.

    (b)$\Rightarrow$(a). Without loss of generality and by $G$-invariance, translate $q$ so that $x\in \arg\min_{p\in [x]}\|q-p\|$. By \cref{prop.principal stratum B}, a lift of $\gamma^r$ is given by the straight line $[x,q]$ and it holds that $g^{-1}z\in (q,x]$ for some $g\in G$. Then by \cref{prop.daduk interval in Omega} applied to the interval $(q,x]$, it follows that $g^{-1}z\in V_x$ and so $z\in Q_x$ as desired.
\end{proof}

If the geodesic metric of $[P(G)]$ were complete, the last statement of \cref{lem.Voronoi char minimal geo ext} would be symmetric in $[x]$ and $[z]$. This is due to the remarkable symmetry of the cut locus for metrically complete Riemannian manifolds (Scholium~3.78 in~\cite{GallotHL:90}.) Although $[P(G)]$ is almost never metrically complete, we are still able to prove the desired symmetry for all Riemannian manifolds with the help of additional hypotheses which $[P(G)]$ satisfies. Since the proof is an adaptation of the arguments for Proposition~13.2.2 in~\cite{DoCarmo:92}, we postpone it to the appendix.

\begin{lemma}
    \label{lem.manifold unique neighborhoods}
    Let $M$ be a Riemannian manifold and let $[x],[z]\in M$. Suppose that the following hypotheses, denoted $H_M([x],[z])$, hold:
    \begin{itemize}
        \item[(i)] There exists a minimal geodesic $\gamma^r$ joining $[x]$ to some $[q]\in M$ such that $[z]$ lies in the interior of $\gamma^r$.
        \item[(ii)] There exists a neighborhood $U$ of $[x]$ such that for every $[p]\in U$, there is a minimal geodesic joining $[z]$ to $[p]$.
        \item[(iii)] For any sequence $[x_i]\to [x]$ and any sequence $\sigma_i$ of minimal geodesics joining $[z]$ to $[x_i]$, it holds that $\sigma_i\rightarrow \gamma$ pointwise in $M$ and $\sigma_i'(0)\to \gamma'(0)$ in $T_{[z]}M$, where $\gamma$ is the geodesic joining $[z]$ to $[x]$ given by the reverse parameterization of $\gamma^r|_{[x]\to [z]}$. 
    \end{itemize}
    Then, $\gamma$ is a minimal geodesic joining $[z]$ to $[x]$ and it remains minimizing shortly beyond $[x]$.
\end{lemma}
\begin{proof}
See \cref{app.cut point symmetry}.
\end{proof}

\begin{lemma}\label{lem.PG manifold hypo}
Suppose $G\leq \Oname(d)$ is compact. For any $x,z\in P(G)$, the Riemannian manifold $[P(G)]$ satisfies the hypotheses $H_{[P(G)]}([x],[z])$ given in \cref{lem.manifold unique neighborhoods}, provided $z \in Q_x$.
\end{lemma}
\begin{proof}
    Since $z\in Q_x$, there exists $g\in G$ such that $z \in g\cdot V_x = V_{gx}$. By \cref{lem.Voronoi char minimal geo ext} and its proof, the first hypothesis holds with $\gamma^r|_{[x]\to [z]}$ given by the image of $[gx,z]$ under the orbit map. The second hypothesis follows from \cref{prop.principal stratum B} which entails that every pair of orbits in $[P(G)]$ are joined by some minimal geodesic in $[P(G)]$.
    
    For the third hypothesis, consider any sequence $[x_i]\to [x]$ and any sequence $\sigma_i$ of minimal geodesics joining $[z]$ to $[x_i]$. By \cref{prop.principal stratum B} and for each $i$, choose $g_i\in G$ so that $g_ix_i \in \arg\min_{q\in [x_i]}\|q-z\|$ and $\sigma_i$ are images of $[z,g_ix_i]$ under the orbit map. We claim that $g_ix_i\to gx$. Once this is established, the straight lines $[z,g_ix_i]$ and their derivatives at $z$ converge pointwise to the straight line $[z,gx]$ and its derivative at $z$, respectively; then these convergences descend, under the orbit map, to the desired limits $\sigma_i \to \gamma$ and $\sigma_i'(0)\to \gamma'(0)$.

    It only remains to prove the claim $g_ix_i\to gx$. Since the sequence is bounded, it suffices to show that every convergent subsequence converges to $gx$. We fix such a convergent subsequence with indices $m_i$. Since $[x_{m_i}]\to [x]$, we have $g_{m_i}x_{m_i}\to p$ for some $p\in [x]$, and we need to show that $p=gx$. By the definition of $g_i$, we have that $\|g_ix_i-z\| \leq \|hg_ix_i-z\|$ for each $i$ and each $h\in G$. Taking limits in $m_i$, we obtain that $\|p-z\|\leq \|hp-z\|$ for each $h\in G$. Thus, $p\in \arg\min_{q\in [x]}\|q-z\| = \{gx\}$, as desired.
\end{proof}

We are now ready to give an almost immediate proof of \cref{lem.PG char}.

\begin{proof}[Proof of \cref{lem.PG char}]
(a)$\Rightarrow$(b). By \cref{lem.PG easier char,lem.Voronoi char minimal geo ext,lem.manifold unique neighborhoods,lem.PG manifold hypo}, $z\in V_x\subseteq Q_x\cap U_x$ implies $x\in Q_z\cap U_z = V_z$ as desired. 

(b)$\Rightarrow$(a). The proof is identical to the (b)$\Rightarrow$(a) case in \cref{lem.PG easier char}, with each $U_x$ and $U_z$ replaced with $V_x$ and $V_z$, respectively.
\end{proof}

\subsection{Geometric Characterization of Regular Orbits}
\label{sec.RG char proof}
This section is dedicated to proving \cref{lem.regular char}. We proceed with similar arguments as in the previous section. To enhance readability and help absorb the core ideas, we recommend skipping the proofs of \cref{lem.desingularization,lem.hypo desing,lem.Voronoi global implies RGRG,lem.Voronoi char minimal geo ext RGRG} on the first read.

The first challenge we encounter is the non-manifold nature of $[R(G)]$. To address this issue, we desingularize the space and localize the analysis to an open neighborhood of a Voronoi cell. The following lemma could be seen as an analogue of \cref{prop.principal stratum}. Since the proof is long and technical, we postpone it to the appendix.

\begin{lemma}
    \label{lem.desingularization}
    Suppose $G\leq \Oname(d)$ is compact and fix $x\in R(G)$. There exists an embedded submanifold $S$ of $G$ passing through the identity $e\in G$ such that each of the following statements holds:
    \begin{lemenum}
        \item \label{lem.desingularization S sym and trans} $S=S^{-1}$ and $S\cap G_x=\{e\}$.
        \item \label{lem.desingularization disjoint fibers}For distinct $p \neq q\in V_x$, $S\cdot p\cap S\cdot q = \varnothing$.
        \item \label{lem.desingularization open SVX} $S\cdot V_x$ is an open subset of $\mathbb R^d$.
        \item \label{lem.desingularization Mx smooth structure} Let $M_x$ denote the space of equivalence classes $\{S\cdot p\}_{p\in V_x}$ equipped with the quotient topology. Then there exists a smooth structure and a Riemannian metric on $M_x$ such that the quotient map $\pi\colon S\cdot V_x\to M_x$ is a smooth Riemannian submersion.
        \item\label{lem.desingularization embedded fiber} For each $p\in V_x$, $S\cdot p$ is an open subset of $G\cdot p$.
        \item \label{lem.desingularization WP}For each $p\in V_x$, there exists an open neighborhood $W_p$ of $x$ in $V_x$ such that $\pi(W_p)$ is open in $M_x$ and for $q\in W_p$, the minimal geodesics joining $\pi(p)$ to $\pi(q)$ in $M_x$ exist and are precisely the $\pi$-images of straight line distance minimizers from $p$ to $\overline S\cdot q$, all of which lie in $S\cdot V_x$.
    \end{lemenum}
\end{lemma}
\begin{proof}
See \cref{app.desingularization proof}. We encourage the reader to glance at the figures therein to get a visual idea of what $S\cdot V_x$ and $M_x$ look like.
\end{proof}

This allows us to prove an analogue of \cref{lem.PG manifold hypo}.

\begin{lemma}
    \label{lem.hypo desing}
    Suppose $G\leq \Oname(d)$ is compact and fix $x\in R(G)$. Take $S$, $M_x$ and $\pi$ as in \cref{lem.desingularization}. The hypotheses $H_{M_x}(\pi(x),\pi(z))$ of \cref{lem.manifold unique neighborhoods} hold for each $z\in V_x$.
\end{lemma}
\begin{proof}
    By the openness of $V_x$ and its star convexity at $x$, as shown in \cref{lem.UV facts intersection V open in normal}, there exists $q\in V_x$ such that $[z,x]\subseteq (q,x]\subseteq V_x$. Then by definition of $V_x$, $[q,x]$ is the unique distance minimizing straight line joining $q$ to $\overline Sx$, and it contains the unique distance minimizing straight line $[z,x]$ joining $z$ to $\overline Sx$. By \cref{lem.desingularization WP} and since $x\in W_q$, we obtain that the $\pi$-image of the straight line $[q,x]$ satisfies the first hypothesis. The second hypothesis is immediate by considering the neighborhood $W_z$ given by \cref{lem.desingularization WP}. The third hypothesis follows by taking limits of distance minimizing straight line lifts $[z,g_ix_i]$, where $g_ix_i\in W_z$, to the unique distance minimizing straight line lift $[z,x]$ as we did in the proof of \cref{lem.PG manifold hypo}.
\end{proof}
We also have two analogues of \cref{lem.Voronoi char minimal geo ext}. The first concerns the (global) open Voronoi cell.
\begin{lemma}
    \label{lem.Voronoi global implies RGRG}
    Suppose $G\leq \Oname(d)$ is compact and fix $x\in R(G)$. Take $S$, $M_x$ and $\pi$ as in \cref{lem.desingularization}. Then for each $z\in V_x$, there exists $q\in V_x$ and a minimal geodesic $\gamma$ in $M_x$ joining $\pi(x)$ to $\pi(q)$ such that $\pi(z)$ lies in the interior of $\gamma$.
\end{lemma}
\begin{proof}
    By the openness of $V_x$ and its star convexity at $x$, as shown in \cref{lem.UV facts intersection V open in normal}, there exists $q\in V_x$ such that $[z,x]\subseteq (q,x]\subseteq V_x$. In particular, $\arg\min_{p\in Sx}\|q-p\| = \{x\}$ and $x\in W_q$, where $W_q$ is given by \cref{lem.desingularization WP}. By the definition of $W_q$, it follows that $\pi([q,x])$ is a minimal geodesic in $M_x$ joining $\pi(q)$ to $\pi(x)$ and containing $\pi(z)$ in its interior. The reverse parameterization, i.e., $\pi([x,q])$, yields the desired conclusion.
\end{proof}
Next, we have an analogue involving the local Voronoi cell.
\begin{lemma}\label{lem.Voronoi char minimal geo ext RGRG}
    Suppose $G\leq \Oname(d)$ is compact and fix $x\in R(G)$. Take $S$, $M_x$ and $\pi$ as in \cref{lem.desingularization}. For $z\in V_x$, the following statements are equivalent:
    \begin{lemenum}
        \item $x\in V^{loc}_z$.
        \item There exists $q\in \mathbb R^d$ such that $d(q,Sz)=\|q-z\|$ and $x\in [z,q)$.
        \item There exists $q\in V_x$ and a minimal geodesic $\gamma$ in $M_x$ joining $\pi(z)$ to $\pi(q)$ such that $\pi(x)$ lies in the interior of $\gamma$.
    \end{lemenum}
    \end{lemma}
    \begin{proof}
        (a)$\Rightarrow$(b). 
        By the definition of $V_x^{loc}$, there exists an open neighborhood $V$ of $x$ and an open neighborhood $U$ of $z$ such that
        \[\forall q\in V, \ \ \Big|\arg\min_{p\in[z]\cap U}\| p-q\| \Big|= 1.\]
        Since $z\in V_x$, we have $\{z\}=\arg\min_{p\in [z]\cap U}\|p - x\|$. By \cref{prop.daduk interval in Omega}, there exists an open neighborhood $Y$ of $(x,z]$ such that $\{z\}=\arg\min_{p\in [z]\cap U}\|n - p\|$ for each $n\in Y\cap N_{z}$. By \cref{prop.daduk Omega normal neighborhood}, it follows that $x\in Y$. Consequently, there exists $\varepsilon > 0$ such that $q:= x + \varepsilon(z-x) \in Y\cap N_z$. Then $x\in [z,q)$ and $d(q,[z]\cap U) = \|q-z\| = d(q, Sz\cap U)$.

        It remains to show that $d(q,Sz\cap U) = d(q,Sz)$ for small $\varepsilon$. If such $\varepsilon$ does not exist, then $d(q_n,Sz\cap U^c) = d(q_n,Sz)$ for a sequence $\varepsilon_n\to 0$. By taking limits, we obtain that $d(x,Sz\cap U^c) = d(x,Sz)$, which is absurd since $\arg\min_{p\in Sz}\|x-p\| = \{z\}\subseteq Sz\cap U$.
    
        (b)$\Rightarrow$(a). This is immediate by \cref{prop.daduk interval in Omega} and the fact that $Sz = [z]\cap U$ for some open neighborhood $U$ ($Sz$ is an open subset of $[z]$ by \cref{lem.desingularization embedded fiber}.)
    
        (b)$\Rightarrow$(c). Since $S=S^{-1}\subseteq \Oname(d)$ by~\cref{lem.desingularization S sym and trans}, it holds that
        \[d(z,Sq) = d(S^{-1}z,q) = d(q,Sz) = \|q-z\|.\]
        By \cref{prop.daduk interval in Omega}, the last equality remains true when taking $q$ close enough to $x$ so that $q\in W_z$ and $x\in [z,q)$, where $W_z$ is given in \cref{lem.desingularization WP}. Then the $\pi$-image of the straight line $[z,q]$ minimizing distance from $z$ to $Sq$ yeilds the desired conclusion.

        (c)$\Rightarrow$(b). By assumption, say $\pi(x)$ lies in the interior of a minimal geodesic $\eta$ joining $\pi(z)$ to some $\pi(q)$, where $q\in V_x$. By \cref{lem.desingularization}, $\pi(W_z)$ is open and $\pi^{-1}(\pi(W_z))\cap V_x = W_z$. As such, we may take $\pi(q)$ close enough to $\pi(x)$ so that $\pi(q)\in \pi(W_z)$ and $q\in W_z$.
        
        Then by \cref{lem.desingularization} and since there is a unique minimal geodesic $\pi([z,x])$ joining $\pi(z)$ to $\pi(x)$, the horizontal lift of $\eta$ initiating from $z$ is given by a straight line $[z,uq]$ which contains $[z,x]$, where $u\in S$ and $\|uq-z\| = d(z,Sq) = d(q,Sz)$. Then $uq\in N_x\cap [q] = V_x\cap [q]$ since the straight line $[z,x]\subseteq N_x$ is contained in $[z,uq]$. As such, $q \in V_x\cap V_{u^{-1}x}$ and so $u\in S\cap G_x = \{e\}$ by \cref{lem.desingularization S sym and trans} and \cref{lem.UV facts intersection}. The desired conclusion is satisfied by $uq=q$.
    \end{proof}
    
We are now ready to give an almost immediate proof of \cref{lem.regular char}.
\begin{proof}[Proof of \cref{lem.regular char}]
(a)$\Rightarrow$(b). This is immediate by combining \cref{lem.hypo desing,lem.Voronoi global implies RGRG,lem.Voronoi char minimal geo ext RGRG}.

(b)$\Rightarrow$(a). As noted after \cref{def.regular} and as shown in \cref{lem.Q invariant,lem.Q open dense}, $R(G)$ and $Q_x = G\cdot V_x$ are $G$-invariant, open and dense. As such, we are able to pick $z\in R(G)\cap V_x$. Then it holds that $G_xz \subseteq R(G) \cap V_x$ because $R(G)$ and $V_x$ are $G_x$-invariant. Moreover, if $gz \in V_x$, then $z \in V_x\cap V_{g^{-1}x}$ which means $g\in G_x$ by \cref{lem.UV facts intersection}. As such,
\[G_xz = R(G)\cap V_x\cap [z].\] 
By assumption, we obtain that $x\in V_{hz}^{loc}$ for each $h\in G_x$. By definition of $V_{hz}^{loc}$ and by \cref{lem.UV facts Gx char}, we obtain that the set $G_xz= \arg\min_{q\in [z]}\|q-x\|\subseteq V_x$ is discrete and $G_x$-transitive. Moreover by \cref{prop.stab subset}, we have that $G_{z}\leq G_x$. By transitivity and the orbit stabilizer theorem, we deduce that $G_x/G_z$ is discrete and hence finite (here, we view $G_z$ as the stabilizer of $z\in G_xz$ under the transitive action of $G_x$.) In particular, $\dim(G_x)=\dim(G_z)$. By the orbit stabilizer theorem, we get that
\[\dim(G\cdot x) = \dim(G)-\dim(G_x) = \dim(G)-\dim(G_z)=\dim(G\cdot z).\]
Since $z\in R(G)$, we obtain that $x\in R(G)$ as desired.
\end{proof}
\section{Local Bilipschitzness at Regular Orbits}\label{sec.local lower stability}
In this section, we provide proofs of the main results \cref{thm.regular lower lip generic,thm.almost free actions}. The core intermediate result is given by the following lemma, whose proof is long and technical and thus postponed to the appendix.

\begin{lemma}
    \label{lem.regular local avoidance}
    Suppose $G\leq \operatorname{O}(d)$ is compact and define $c:= d- \max_{x\in \mathbb R^d}\dim([x])$.
    For $z_1,\dots,z_n\in \mathbb R^d$, denote the corresponding max filter bank by $\Phi\colon \mathbb R^d/G\to \mathbb R^n$. Then the set
    \begin{equation}
        \label{eq.R def}
    R := \Big\{\{z_i\}_{i=1}^n \in (\mathbb R^d)^n: \Phi \text{ fails to be locally lower Lipschitz at every $x\in R(G)$}\Big\}
    \end{equation}
    is semialgebraic, and it holds that
    \begin{equation}
        \label{eq.R bound}
        \dim(R)\leq nd - 1 - \left(\left\lceil\frac{n}{\chi(G)}\right\rceil - 2c + 1\right).
    \end{equation}
\end{lemma}
\begin{proof}
    See \cref{app.regular local avoidance proof}.
\end{proof}
Here, a map $f\colon \mathbb R^d/G \to \mathbb R^n$ is said to be locally lower Lipschitz at $x\in \mathbb R^d$ if it is lower Lipschitz when restricted to a neighborhood of $x$, with respect to the quotient distance induced by $G$. The proof of \cref{thm.regular lower lip generic} follows immediately from the above lemma and the fact that max filter banks are (globally) upper Lipschitz. Next, an adjustment of the proof of Theorem~5(c) in~\cite{MixonQ:22} gives the following proposition on the generic injectivity of max filter banks.
\begin{proposition}
    \label{prop.injective}
    If $G\leq \Oname(d)$ is compact with $c:= d-\max_{x\in\mathbb R^d}\dim([x])$, then for generic $z_1,\ldots,z_n\in\mathbb{R}^d$, the max filter bank $[x]\mapsto\{\llangle [z_i],[x]\rrangle\}_{i=1}^n$ is injective provided $n\geq 2c$.
\end{proposition}

We are now ready to give an almost immediate proof of \cref{thm.almost free actions}.

\begin{proof}[Proof of \cref{thm.almost free actions}]
    Fix an arbitrary max filter bank $\Phi([x]):=\{\llangle [x],[z_i]\rrangle\}_{i=1}^n$.  Since $\Phi$ is $\|\{z_i\}_{i=1}^n\|_F$-Lipschitz, it fails to be bilipschitz if and only if it fails to be lower Lipschitz. This occurs if and only if there exist sequences $[x_j]\neq [y_j]$ such that 
\begin{equation}
    \label{eq.corr proof conv}
    \frac{\Phi([x_j])-\Phi([y_j])}{d([x_j],[y_j])} \rightarrow 0.
\end{equation}
Since \eqref{eq.corr proof conv} is symmetric and invariant under simultaneous positive dilations of $x_j$ and $y_j$, we may without loss of generality assume that $\|y_j\|\leq \|x_j\|=1$. By taking subsequences, there exist $x,y\in\mathbb R^d$ such that $[x_j]\to [x]$ and $[y_j]\to [y]$. Notably, since $\|x\|=1$, we have $x\in \mathbb R^d-\{0\}\subseteq R(G)$. 

In the case $[x]\neq [y]$, it follows that $d([x_j],[y_j])\gg 0$, and so $\Phi([x])=\Phi([y])$, meaning that $\Phi$ fails to be injective.

In the case $[x] = [y]$, it follows that $\Phi$ fails to be locally lower Lipschitz at $x\in R(G)$. 

As such, if $\Phi$ is injective and locally lower Lipschitz at every $x\in R(G)$, then $\Phi$ is bilipschitz. The result now follows by combining \cref{thm.regular lower lip generic,prop.injective}.
\end{proof}
    
\section{Discussion}
\label{sec.discussion}
In this paper, we demonstrated that sufficiently many generic templates ensure max filter banks are bilipschitz when all nonzero orbits of $G\leq \Oname(d)$ have maximal dimension, i.e., lie in $R(G)$. To achieve this, we established that max filter banks are generally locally bilipschitz on $R(G)$.

This work leaves open two key questions.
\begin{problem}$ $
\begin{itemize}
    \item[(a)]  Is every max filter bank bilipschitz provided enough generic templates?
    \item[(b)]  Is every injective max filter bank bilipschitz?
\end{itemize}
\end{problem}

Indeed, the second question is much stronger than the first. To address the first, one would need to generalize \cref{lem.regular local avoidance} to establish local lower Lipschitzness at nonregular points. A specific, unresolved example is the real irreducible representation of $\operatorname{SO}(3)$ within $\Oname(7)$.

For the stronger question, a counterexample arises if an injective max filter bank is found such that the image of $D$ in \eqref{eq.D defn} is not closed. The smallest example worth investigating involves the circle group $S^1$ on $\mathbb C^2\times \mathbb R$ given by $\theta \mapsto \operatorname{diag}\{e^{i\theta},e^{i\theta}, 1\}$. Conversly, proving the affirmative would likely require adapting the techniques in \cite{BalanT:23} to the setting of infinite groups.

In \cref{sec.relevance of results}, we demonstrated how max filter banks offer a theoretically desirable tool for a nearest neighbor problem in cryo-EM. Testing the hypothesis numerically could uncover potential improvements over the bispectrum embedding used in \cite{ZhaoS:14}.

In \cref{sec.voronoi BIG SECTION}, we derived geometric characterizations of regularity and principality using Voronoi cell decompositions. Extending these characterizations to include nonregular points remains an intriguing direction for future research.

\section*{Acknowledgments}

The author thanks Dustin G. Mixon for helpful remarks, as well as Efstratios Tsoukanis for enlightening discussions.

\appendix

\section{Riemannian geometric arguments}
\label{app.riemmanian arguments main}
This section is dedicated to proving \cref{lem.manifold unique neighborhoods,lem.desingularization}. Before providing their proofs in \cref{app.cut point symmetry,app.desingularization proof}, respectively, we begin by stating necessary preliminary results in \cref{app.riemannian prelim}.
\subsection{Preliminaries}
\label{app.riemannian prelim}
The first preliminary we need is for \cref{app.cut point symmetry}, and it concerns subgeodesics of minimal geodesics. It is left as an exercise in Corollary~2.111 in~\cite{GallotHL:90}. We view the result as well-known, but we provide a proof for the sake of convenience.

\begin{proposition}
    \label{prop.geodesic minimal whole prop}
    Let $M$ be a Riemannian manifold and suppose that $c\colon [a,b]\to M$ is a minimal geodesic joining $c(a)$ to $c(b)$. Then $c|_I$ is minimal over every subinterval $I\subseteq [a,b]$.
\end{proposition}
\begin{proof}
Let $d$ denote the geodesic distance of $M$, and let $I = [i_0,i_1]\subseteq [a,b]$ be a subinterval. Let $\eta\colon I\to M$ be a unit speed piecewise $C^1$ curve joining $c(i_0)$ to $c(i_1)$ such that its length satisfies $L(\eta)\leq L(c|_I)$. By the traingle inequality and the definition of the geodesic distance, we get that 
\[
\begin{aligned}
    L(c) &= d(c(a),c(b))\\
    &\leq d(c(a),c(i_0))+d(c(i_0),c(i_1))+d(c(i_1),c(b))\\
    &\leq L(c|_{[a,i_0]}) + L(\eta) + L(c|_{[i_1,b]})\\
    &\leq L(c|_{[a,i_0]}) + L(c|_I) + L(c|_{[i_1,b]})\\
    &=L(c).
\end{aligned}
\]
As such, $L(\eta) = L(c|_I)$ and so $c|_I$ is minimal as desired.
\end{proof}

Next, we give an essential preliminary for \cref{app.desingularization proof}. It is a collection of statements regarding the orthogonal slice and tubular neighborhood geometry of compact Lie group isometric actions on manifolds. Before stating the preliminary, we define what we mean by a smooth isometric action.
\begin{definition}
    Let $G$ be a Lie group with identity $e$ and let $M$ be a Riemannian manifold. A smooth map $\mu\colon G\times M\to M$ is called a (left) smooth isometric action of $G$ on $M$ if:
    \begin{itemize}
    \item[(a)]$\mu(e,x)=x$, for all $x\in M$.
    \item[(b)]$\mu(g_1,\mu(g_2,x)) = \mu(g_1g_2,x)$, for all $g_1,g_2\in G$ and $x\in M$.
    \item[(c)]$\mu(g,\cdot)\colon M\to M$ is an isometry of $M$, for all $g\in G$.   
    \end{itemize}
\end{definition}

In the following proposition, we denote $g\cdot x := \mu(g,x)$.
\begin{proposition}[Tubular Neighborhood Theorem]
    \label{prop.slice theorem}
    Let $G$ be a compact Lie group acting smoothly and isometrically on a Riemannian manifold $M$, and fix any $x\in M$. Then $G\cdot x$ is an embedded submanifold of $M$, and there exists an open neighborhood $B$ of $0$ in $(T_x(G\cdot x))^\perp$ such that each of the following statements holds:
    \begin{propenum}
    \item\label{prop.slice is manifold} $S_x := \exp_{x}(B)$ is a $G_x$-invariant embedded submanifold of $M$. 
    \item  \label{prop.slice tube is open}$G\cdot S_x$ is an open neighborhood of $G\cdot x$.
    \item \label{prop.slice tube diffeo}Let $G\times_{G_x} S_x$ denote the orbit space of the smooth and free $G_{x}$-action on $G\times S_x$ given by $h\cdot (g,s):= (gh^{-1},hs)$ for $h\in G_x$. Then the map $\Psi_x\colon G\times_{G_x} S_x\to G\cdot S_x$, induced by $(g,s)\mapsto g\cdot s$, is a diffeomorphism.
    \item \label{prop.slice technical trans diffeo}Suppose that $H$ is an embedded submanifold of $G$ transverse to $G_x$ at $e$, i.e., $T_eH\cap T_eG_x = \{0\}$ and $T_eH + T_eG_x = T_eG$. Moreover, suppose that the multiplication map $H\times G_x\to H\cdot G_x$ is a diffeomorphism. Then the map $F\colon H\times S_x \to HG_x \times_{G_x} S_x$, induced by $(h,s)\mapsto (h,s)$, is a diffeomorphism.
    \end{propenum}
\end{proposition}
For our purposes, we call any $S_x$ which satisfies the statements of \cref{prop.slice theorem} an \textbf{orthogonal slice} at $x$ to the left action of $G$ on $M$. This differs from the standard approach in the literature (e.g. Definition~3.47 in~\cite{AlexandrinoB:15}.)

Before giving a proof by references, we state a few remarks. In the proposition above, $G\times_{G_x} S_x$ is equipped with the unique smooth structure that turns the corresponding free actoin's orbit map $G\times S_x\to G\times_{G_x}S_x$ into a smooth submersion. This is a special case of Theorem~3.34 in~\cite{AlexandrinoB:15} which applies to general proper free actions. We also note that the name `Tubular Neighborhood Theorem' stems from viewing $G\cdot S_x$ as an open `tubular' neighborhood of $G\cdot x$.

\begin{proof}[Proof of \cref{prop.slice theorem}]
    The assertion that $G\cdot x$ is an embedded submanifold of $M$ is given by Proposition~3.41 in~\cite{AlexandrinoB:15} (the result there is for so called proper actions, but we note that every smooth action of a compact Lie group is proper.) The statements of (a), (b) and (c) follow from the statement and proof of Theorem~3.57 as well as the paragraph preceeding Definition~3.72 in~\cite{AlexandrinoB:15}. Lastly, (d) follows from Claim~3.52 in~\cite{AlexandrinoB:15} applied to $P=G$, $F=S_x$, $S=H$ and $U=HG_x$.
\end{proof}

\subsection{Proof of Lemma~\ref{lem.manifold unique neighborhoods}}\label{app.cut point symmetry}

\begin{figure}[t]
    \centering
    \includegraphics[scale=1.2]{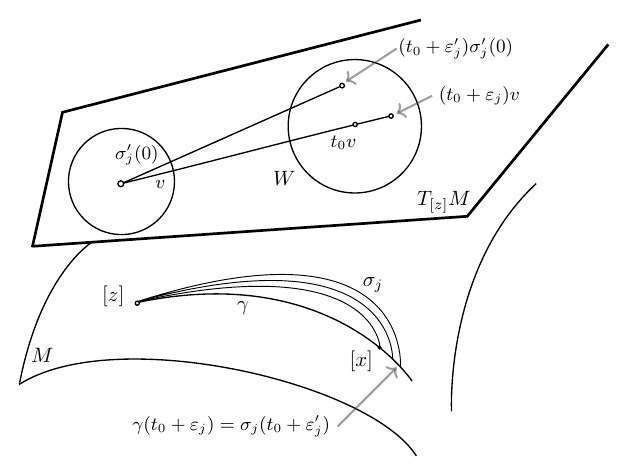}
    \caption{This figure is an aiding illustration for \cref{app.cut point symmetry}. It is an adaptation of Figure~13.2.1 in~\cite{DoCarmo:92}.}
    \label{fig:cutlocusaiding}
  \end{figure}

Let $d$ denote the geodesic distance of $M$, and let $\gamma^r:[0,d([x],[q])]\to M$ be the minimal geodesic joining $[x]$ to $[q]$, as given by the first hypothesis. By \cref{prop.geodesic minimal whole prop}, its restriction $\gamma^r|_{[x]\to [z]}$ is a minimal geodesic joining $[x]$ to $[z]$.

Let $\gamma:[0,d([z],[x])]\to M$ be the minimal geodesic joining $[z]$ to $[x]$ defined as the reverse parameterization of $\gamma^r|_{[x]\to [z]}$. We aim to show that $\gamma$ remains minimizing shortly beyond $[x]$.

We begin by reframing the aim in terms of the exponential map. Since $\gamma$ is a geodesic, we have $\gamma(t) = \operatorname{exp}_{[z]}(tv)$, where $v:= \gamma'(0)$ and $\exp\colon TM\to M$ is the Riemannian exponential map of $M$. By Proposition~5.19 in~\cite{Lee:18}, the exponential map is smooth and its domain is open. It follows that $\operatorname{exp}_{[z]}(tv)$ is defined in a neighborhood of $d([z],[x])v\in T_{[z]}M$.

We seek to show that $\operatorname{exp}_{[z]}(tv)$ remains minimizing for $t$ shortly beyond $d([z],[x])$. Suppose otherwise for the sake of contradiction. We closely follow the proof of Proposition~13.2.2 in~\cite{DoCarmo:92} which treats the case of complete manifolds. Put $t_0 := d([z],[x])$ and let $\{t_0+\varepsilon_i\}$ be a sequence in which $\varepsilon_i >0$ and $\varepsilon_i\to 0$. By the second hypothesis and for large $i$, there exists a sequence of minimizing geodesics $\sigma_i$ joining $[z]$ to $\operatorname{exp}_{[z]}((t_0+\varepsilon_i)v)$, and $\sigma_i'(0)\in T_{[z]}M$ is the corresponding sequence of tangent vectors at $[z]$. By the third hypothesis, $\sigma_i \rightarrow \gamma$ pointwise in $M$ and $\sigma_i'(0)\rightarrow \gamma'(0)$ in $T_{[z]}M$. See \cref{fig:cutlocusaiding} for an aiding illustration.

 We show that $d \operatorname{exp}_{[z]}$ is singular at $t_0\gamma'(0)$. Suppose otherwise for the sake of contradiction. Then, there exists a neighborhood $W$ of $t_0\gamma'(0)$ such that $\operatorname{exp}_{[z]}\big|_W$ is a diffeomorphism. By definition of $\sigma_j$, $\gamma(t_0+\varepsilon_j)=\sigma_j(t_0+\varepsilon_j')$, where $\varepsilon_j'\leq \varepsilon_j$ because $\sigma_j$ is minimizing. Take $\varepsilon_j$ sufficiently small so that $(t_0+\varepsilon_j')\sigma_j'(0)$ and $(t_0+\varepsilon_j)\gamma'(0)$ belong to $W$. Then,
\[\operatorname{exp}_{[z]}((t_0+\varepsilon_j)\gamma'(0)) = \operatorname{exp}_{[z]}((t_0+\varepsilon_j')\sigma_j'(0)).\]
Thus $(t_0+\varepsilon_j)\gamma'(0)=(t_0+\varepsilon_j')\sigma_j'(0)$, and so $\gamma'(0)=\sigma_j'(0)$. This contradicts the assumption that $\gamma$ is no longer minimizing beyond $t_0$. As such, there exists nonzero $u\in T_{[z]}M$ such that $d_{t_0\gamma'(0)}\operatorname{exp}_{[z]}u = 0$. By Corollary~3.46 and Definition~3.72 in~\cite{GallotHL:90}, it holds that $[z]$ and $[x]$ are \textit{conjugate} along $\gamma$.

Now, let $\gamma^{r}\colon [0,d([x],[z])]\to M$ be the unique minimal geodesic joining $[x]$ to $[z]$. Since $\gamma^r$ is the reverse parameterization of $\gamma$, $[x]$ and $[z]$ are also conjugate along $\gamma^r$. However, by the first hypothesis, $\gamma^r$ remains minimizing shortly beyond $[z]$. This contradicts the fact that geodesics fail to be minimizing beyond conjugate points (Theorem~3.73(ii) in \cite{GallotHL:90}.)
    \subsection{Proof of Lemma~\ref{lem.desingularization}}
    \label{app.desingularization proof}
    %Since $G$ acts freely on some orbit $G\cdot y\subseteq \mathbb R^d$ and $x\in R(G)$, the orbit stabilizer theorem entails that
    %    \[\dim(G/G_x) \leq \dim(G) = \dim(G\cdot y) \leq \dim(G\cdot x) = \dim(G/G_x).\]
    %    As such, $\dim(G/G_x)=\dim(G)$, and so $G_x$ is finite.
    This section is dedicated to the proof of \cref{lem.desingularization}, which is both long and technical. To enhance readability and organization, the proof is divided into subsections, each corresponding to a specific part (a)-(f) of the lemma, in order. Within each subsection, we present a sequence of claims, each accompanied by proof.
    
    Notably, the last claim in each subsection is the key result used in subsequent subsections; all other claims within the subsection serve as intermediate steps and are only relevant locally.     Once the last claim in a subsection is proven, it completes the proof of the corresponding part of the lemma. To aid in visualization, we provide diagrams throughout the proof. We hope that this structure allows the reader to follow the argument linearly while minimizing the need to reference earlier claims from previous subsections.
    
    Throughout, we view $G$ as a compact Lie subgroup and embedded submanifold of $\operatorname{O}(d)$. In particular, its Lie exponential agrees with the matrix exponential. We denote its Lie algebra by $\mathfrak g\subseteq \mathbb R^{d\times d}$ and its identity by $e:= \operatorname{id}_{\mathbb R^d}$. For a compact subgroup $H\leq G$, we denote by $H^0$ the connected component of $H$ which contains $e$. 
    
    \subsubsection{Proof of Lemma~\ref{lem.desingularization S sym and trans}}

    We begin by constructing the desired submanifold $S$. The following is a stronger version of \cref{lem.desingularization S sym and trans}.
    \begin{claim}[Construction of $S$]
        \label{claim.construction of S}
        There exists an embedded submanifold $S\subseteq G$ passing through $e$ such that each of the following statements holds: 
        \begin{claimenum}
        \item \label{claim.S symmetric}$S$ is symmetric, i.e., $S^{-1}=S$.
        \item \label{claim.S transverse}For each $g\in G$, $gSg^{-1}$ is transverse to $gG_xg^{-1} = G_{gx}$ at $e$, i.e., $T_e(gSg^{-1}) \cap T_e G_{gx} = \{0\}$ and $T_e(gSg^{-1}) + T_e G_{gx} = \mathfrak g$.
        \item \label{claim.S with stabilizer intersect}$\overline{S^5}\cap G_x \subseteq G_x^0$ and $\overline S\cap G_x = \{e\}$.
        %\item $S^2h_1 \cap S^2h_2 = \varnothing$ for each $h_1, h_2 \in G_x$ with $h_1G_x^0 \neq h_2 G_x^0$.
        \item \label{claim.S exp diffeo}There exists an open neighborhood $T$ of $0$ in $T_eS$ such that $\exp|_T\colon T\to S$ is a diffeomorphism, where $\exp\colon \mathfrak g\to G$ is the matrix exponential.
        \item \label{claim.S right mult diffeo}The multiplication map $S\times G_x\to S\cdot G_x$ is a diffeomorphism. In particular, $S\cdot G_x$ is an open subset of $G$.
        \end{claimenum}
    \end{claim}

        \begin{figure}[t]
        \centering
        \includegraphics[scale=0.8]{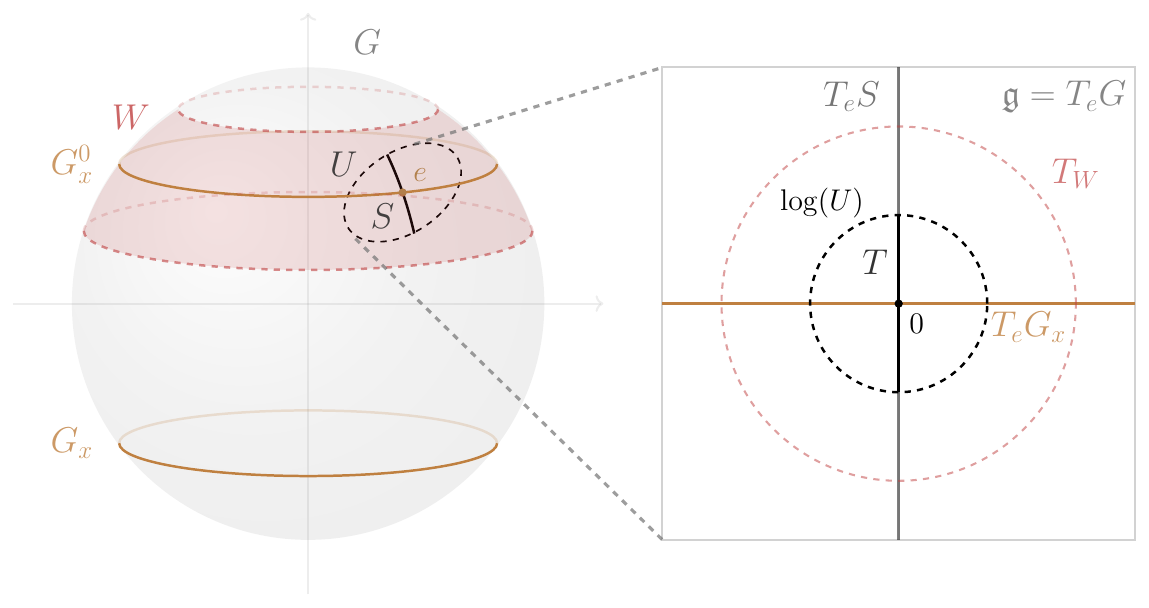}
        \caption{This figure is an aiding illustration for the proof of \cref{claim.construction of S}. For brevity, we define $\log(U) := \exp|_{T_W}^{-1}(U)$. The dashed lines connecting the left and right halves of the figure are solely for visualizing the transition by magnification from $G$ to its Lie algebra.}
        \label{fig:spheregroup}
      \end{figure}

    \begin{proof}
        See \cref{fig:spheregroup} for an aiding illustration.

        \vspace{\claimstepvspace}\noindent \textit{Step 1. We construct a nice neighborhood of the identity $e \in G$.}

        Since $G$ is a Lie group, the $5$-fold multiplication map $m_G\colon G\times G\times G\times G\times G\to G$ is well-defined and continuous. Since $G_x$ has finitely many connected components each of which is compact, there exists an open neighborhood $W$ of $G_x^0$ such that $\overline W \cap G_x = G_x^0$ and such that the matrix exponential map restricts to a diffeomorphism of an open neighborhood $T_W \subseteq \mathfrak g$ of $0$ onto a subset $\exp(T_W)\subseteq W$ that satisfies $\overline{\exp(T_W)}\cap G_x = \exp(\overline{T_W}\cap T_eG_x)$. The latter is possible since $G_x$ is a compact subgroup hence an embedded submanifold of $G$.        

        Then for a simple open neighborhood $W_1\times W_2\times W_3 \times W_4\times W_5\subseteq m_G^{-1}(W)$ which contains $(e,e,e,e,e)$, define $U:=W_1\cap W_2\cap W_3\cap W_4\cap W_5\cap \exp(T_W)$. One immediately observes that $U$ is a neighborhood of $e$, $\overline{U^5}\cap G_x \subseteq \overline{W}\cap G_x = G_x^0$, and $\overline U\subseteq \overline{\exp(T_W)}$.

        \vspace{\claimstepvspace}\noindent \textit{Step 2. We construct $S$.}

        Equip $G$ with a bi-invariant Riemannian metric $\beta$ so that $G_x$ acts freely and smoothly on $G$ by inverted right multiplication isometries. Let $(T_eG_x^0)^\perp$ denote the orthogonal complement of $T_eG_x^0$ in $\mathfrak g$, with respect to $\beta$. By the Tubular Neighborhood Theorem (\cref{prop.slice theorem}), there exists $T$, an open neighborhood of $0$ in $(T_eG_x^0)^\perp\cap \exp|_{T_W}^{-1}(U)$, such that $-T = T$ and $S:= \exp(T) \subseteq U$ is an orthogonal slice at $e$ to the aforementioned action of $G_x$ on $G$.
        
        \vspace{\claimstepvspace}\noindent \textit{Step 3. We verify that $S$ satisfies all the properties we seek.}

         For (a), since $-T=T$, it follows that $S=S^{-1}$. For (c), since $S\subseteq U$ and $\overline U \subseteq \overline{\exp(T_W)}$, it follows that $\overline{S^5}\cap G_x \subseteq G_x^0$ and $\overline S\cap G_x = \{e\}$. For (d), observe that $T\subseteq T_W$.

        Next, we prove (e). Since the aforementioned action of $G_x$ on $G$ is free, \cref{prop.slice tube diffeo} entails that the right inverted (and hence noninverted) multiplication map $S\times G_x\to S\cdot G_x$ is a diffeomorphism.
        
        Lastly, we prove (b). By \cref{prop.slice is manifold}, it follows that $S$ is an embedded submanifold orthogonally transverse to $G_x$ at $e$. Next, the adjoint map $\operatorname{Ad}_g\colon \mathfrak g\to \mathfrak g$, namely the derivative of conjugation $h\mapsto ghg^{-1}$ at the identity, is an isometry of $\mathfrak g$ with respect to the bi-invariant Riemannian metric $\beta$. As such, for each $g\in G$, it holds that $gSg^{-1}$ is orthogonally transverse to $gG_xg^{-1} = G_{gx}$ at $e$ as desired.
    \end{proof}
    \subsubsection{Proof of Lemma~\ref{lem.desingularization disjoint fibers}}

    Before proceeding, we need the following observation regarding stabilizers. It is an analogue of \cref{prop.stab PG to PG}.

    \begin{claim}
        \label{claim.connected comp stab agree}
        For each $y\in R(G)$ and $p\in V_y$, it holds that $G_p\leq G_y$ and $G_p^0=G_y^0$. In particular, $V_y\subseteq R(G)$.
    \end{claim}
    \begin{proof}
        The assertion $G_p\leq G_y$ follows from \cref{prop.stab subset}. In particular, $\dim(G_p)\leq \dim(G_y)$ and $T_eG_p \subseteq T_eG_y$. For the second assertion and since $y\in R(G)$, the orbit stabilizer theorem entails that
        \[\dim(G/G_p) = \dim(G\cdot p)\leq \dim(G\cdot y) = \dim(G/G_y).\]
        As such, $\dim(G_y) \leq \dim(G_p)$ and so $T_eG_y = T_eG_p$. Since $G_y$ and $G_p$ are compact embedded Lie subgroups of $G$, their respective exponential maps are surjective onto their respective connected components. Thus
        \[G_y^0 = \exp(T_eG_y) = \exp(T_eG_p) = G_p^0.\]
    \end{proof}

    The construction of $S$ and the aforementioned claim allow us to prove a stronger version of the statement of \cref{lem.desingularization disjoint fibers}. For the following claim, see \cref{fig:desingfigure} for an illustration.
    
    \begin{claim}[Fiber Disjointness]
    \label{claim.fiber disjoint}
    Let $a$ and $b$ be nonnegative integers such that $a+b \leq 5$, and fix an arbitrary $p\in V_x$. For all $p_1,p_2\in V_p$, if $p_1\neq p_2$, then $\overline{S^a} p_1\cap \overline{S^b} p_2 = \varnothing$. In particular, it holds that $\overline{S^4}p'\cap S\cdot V_x = Sp'$, for all $p'\in V_x$.
    \end{claim}
    \begin{proof}
     We prove the contrapositive. If $u_1 \in \overline {S^a}$ and $u_2\in \overline{S^b}$ are such that $u_1p_1 = u_2p_2$, then $p_1 = u_1^{-1}u_2p_2$ and so $u_1^{-1}u_2 V_p\cap V_p\neq \varnothing$. By \cref{lem.UV facts,prop.stab subset} and \cref{claim.S with stabilizer intersect}, it follows that $u_1^{-1}u_2 \in G_p \cap \overline{S^a}^{-1}\overline{S^b} \subseteq G_x^0$ (since $G_p\leq G_x$ and $\overline{S^a}^{-1}\cdot \overline{S^b} = \overline{S^a}\cdot\overline{S^b} \subseteq \overline{S^5}$.) By \cref{claim.connected comp stab agree}, we get that $u_1^{-1}u_2 \in G_x^0 = G_p^0 = G_{p_2}^0\leq G_{p_2}$ and so $p_1 = u_1^{-1}u_2p_2 = p_2$ as desired.
    \end{proof}

    \begin{figure}[t]
        \centering
        \includegraphics[scale=0.95]{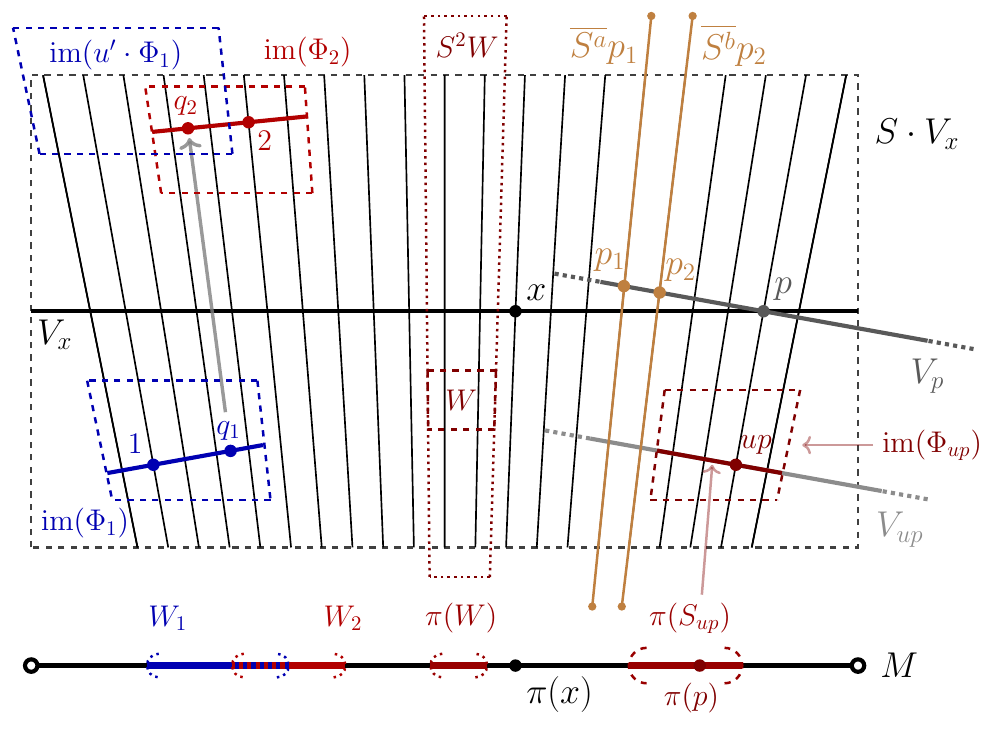}
        \caption{In order from left to right, this figure is an aiding illustration for the proof of \cref{claim.smooth structure} and the statements of \cref{claim.pi is open,claim.fiber disjoint,claim.nice charts SVX}, respectively.}
        \label{fig:desingfigure}
      \end{figure}

    \subsubsection{Proof of Lemma~\ref{lem.desingularization open SVX}}
    We restate and prove \cref{lem.desingularization open SVX} in the following claim.
    \begin{claim}[Open Parent Space]
        \label{claim.openness}
        $S\cdot V_x$ is an open subset of $\mathbb R^d$.
    \end{claim}
    \begin{proof}
     Since $Q_x = G\cdot V_x$ is open by \cref{lem.Q open dense}, it holds that $S\cdot V_x\subseteq Q_x$ is open in $\mathbb R^d$ if and only if $S\cdot V_x$ is open in $G\cdot V_x$. Note that $(SG_x) \cdot V_x = S\cdot V_x$ since $G_x$ fixes $V_x$ by \cref{lem.UV facts equivariance}. Suppose for the sake of contradiction that $S\cdot V_x$ is not open in $G\cdot V_x$. Then there exist sequences $g_n\in G$ and $p_n \in V_x$ such that $g_np_n \to uq \in S\cdot V_x$ yet $g_np_n\notin SG_x\cdot V_x$. In particular, $g_n\notin SG_x$. By compactness of $G$, we may assume $g_n\to g\in G$. Note that $g\notin SG_x$ since $SG_x$ is open in $G$ by \cref{claim.S right mult diffeo}. Additionally, $p_n\to p:=g^{-1}uq \in V_{g^{-1}u x}\cap \overline V_x$. By \cref{lem.UV facts intersection}, we get that $g^{-1}u \in G_x$ and so $g \in uG_x \subseteq SG_x$, a contradiction.
    \end{proof}

    \subsubsection{Proof of Lemma~\ref{lem.desingularization Mx smooth structure}}

    By virtue of \cref{claim.fiber disjoint}, let $M$ denote the space of equivalence classes $\{S\cdot p\}_{p\in V_x}$ equipped with the quotient topology. Let $\pi\colon S\cdot V_x \to M$ denote the quotient map. We build up towards proving that $M$ is a topological manifold that admits a Riemannian structure that makes $\pi$ into a Riemannian submersion. We begin with describing saturations and proving that $\pi$ is an open map. For the following claim, see \cref{fig:desingfigure} for an illustration.
    \begin{claim}[Open Quotient Map]
    \label{claim.pi is open}
    For $W\subseteq S\cdot V_x$, it holds that \[\pi^{-1}(\pi(W))= (S^2\cdot W)\cap S\cdot V_x.\]
     As a consequence, $\pi$ is an open map.
    \end{claim}
    \begin{proof}
        Let $W\subseteq S\cdot V_x$. An element $k \in W$ has the form $k = up$ for some $u\in S$ and $p\in V_x$. Then, $Sp \subseteq SS^{-1}k = S^2k \subseteq S^3p$, which when combined with \cref{claim.fiber disjoint} gives
        \[Sp \subseteq S^2k \cap S\cdot V_x \subseteq S^3p\cap S\cdot V_x = Sp.\]
        As such, we get that $S^2k \cap S\cdot V_x  = Sp = \pi^{-1}(\pi(k))$, and so
        \[\pi^{-1}(\pi(W)) = \cup_{k\in W}\pi^{-1}(\pi(k)) = \cup_{k\in W}S^2k \cap S\cdot V_x = S^2\cdot W \cap S\cdot V_x.\]
        To see that $\pi$ is an open map, let $W\subseteq S\cdot V_x$ be open and observe that
        \[\pi^{-1}(\pi(W)) = (\cup_{u\in S^2} u\cdot W) \cap S\cdot V_x.\]
        The latter set is open in the subspace topology of $S\cdot V_x$ since each $u\in S^2$ is a homeomorphism of $\mathbb R^d$. By definition of the quotient topology, we obtain that $\pi(W)$ is open as desired.
    \end{proof}
    
    As an immediate application, we prove the following claim.

    \begin{claim}
    \label{claim.Hausdorff}
    $M$ is Hausdorff and second countable.    
    \end{claim}
    \begin{proof}
    For each $p,q\in V_x$, we have $\overline Sp\cap \overline Sq = \varnothing$ by \cref{claim.fiber disjoint}. By compactness, these two sets are separated by open neighborhoods $O_p$ and $O_q$ in $\mathbb R^d$. Since $\pi$ is an open map, we get that $\pi(O_p\cap S\cdot V_x)$ and $\pi(O_q\cap S\cdot V_x)$ are open in $M$ and separate $\pi(p)$ and $\pi(q)$. This shows that $M$ is Hausdorff. Second countability is immediate from the fact that $S\cdot V_x$ is second countable and $\pi$ is a continuous surjective open map, as shown in \cref{claim.pi is open}.
    \end{proof}

    In the aim of constructing charts on $M$, we construct nice charts in the parent space $S\cdot V_x$. For the following claim, see \cref{fig:desingfigure} for an illustration.

    \begin{claim}[Nice Charts]
    \label{claim.nice charts SVX}
    For each $u\in S$ and $p\in V_x$, there exists an open neighborhood $S_{up}$ of $up$ in $V_{up}\cap S\cdot V_x$ and an open neighborhood $T_{up}$ of $0$ in $T=\exp^{-1}(S)\subseteq \mathfrak g$ such that each of the following holds:
    \begin{claimenum}
        \item \label{claim.pi sup injective}$\pi|_{S_{up}}$ is injective.
        \item \label{claim.PHI UP lands in SVX}$u\exp (T_{up})u^{-1} \cdot S_{up}\subseteq S\cdot V_x$.
        \item \label{claim.PHI UP construction} The adjoint multiplication map
        $\Phi_{up}\colon T_{up}\times S_{up}\to  u\exp(T_{up})u^{-1}\cdot S_{up}$ given by 
        \[\Phi_{up}(t,q):= u\exp(t)u^{-1}q\] 
        is a diffeomorphism. In particular, $u\exp(T_{up})u^{-1}\cdot S_{up}$ is open in $\mathbb R^d$.
        \item \label{claim.first comm diagram}The following diagram commutes,
        \begin{equation}
            \label{diag.nice Phi}
            \begin{tikzcd}
                u\exp(T_{up})u^{-1}\cdot S_{up} \arrow[r, "\pi"]& M\\
                T_{up}\times S_{up} \arrow[u,"\Phi_{up}"]\arrow[r, "\Pi_{S_{up}}"] & S_{up}\arrow[u,"\pi|_{S_{up}}" swap]
            \end{tikzcd}
        \end{equation}
        where $\Pi_{S_{up}}$ is given by projection onto the component of $S_{up}$.
    \end{claimenum}
    \end{claim}
    \begin{proof}
        
    \noindent \textit{Step 1. We invoke the Tubular Neighborhood Theorem, i.e., \cref{prop.slice theorem}.}

     In our particular case, the theorem entails that at each $up \in S\cdot V_x$, there exists an open $G_{up}$-invariant neighborhood $S_{up}$ of $up$ in $V_{up}\cap S\cdot V_x\subseteq N_{up}$ such that $G\cdot S_{up}$ is an open $G$-invariant tubular neighborhood of $G\cdot up$ and the map $\Psi_{up}\colon G \times_{G_{up}} S_{up}\to G\cdot S_{up}$, induced by multiplication $(g,s)\mapsto gs$, is a diffeomorphism.
    
    %In fact, $G \times_{G_{up}} S_{up}$ is the total space of a smooth fiber bundle
    %\[S_{up}\to G \times_{G_{up}} S_{up} \to G/G_{up},\]
    %where the fiber bundle projection map $\Pi_{up}\colon G \times_{G_{up}} S_{up} \to G/G_{up}$ is induced by $(g,s)\mapsto g \mapsto gG_{up}$.
    
    By \cref{prop.slice technical trans diffeo} and since $uSu^{-1}$ is orthogonally transverse to $G_{up}^0 = uG_x^0u^{-1}$ and since the right multiplication map $uSu^{-1}\times G_{up}\to uS\cdot G_{p}u^{-1}$ is a diffeomorphism as shown in \cref{claim.S right mult diffeo}, the map $F\colon uSu^{-1}\times S_{up} \to uSG_{p}u^{-1}\times_{G_{up}} S_{up}$ defined by $F(r,s):= [(r,s)]$ is a diffeomorphism.
    
    \vspace{\claimstepvspace}\noindent\textit{Step 2. We construct $\Phi_{up}$ via composition and we prove (b) and (c).}

    Recall that $S=\exp(T)$ as in \cref{claim.S exp diffeo}, and define $\phi_{up}\colon T\times S_{up}\to G\cdot S_{up}$ by
    \[\phi_{up}:=\Psi_{up} \circ F \circ (\operatorname{C}_u \circ \exp \times \operatorname{id}_{S_{up}}),\]
    where $\operatorname{C}_u\colon G\to G$ is the diffeomorphism given by conjugation $g\mapsto ugu^{-1}$. Then $\phi_{up}$ is a diffeomorphism of $T\times S_{up}$ onto $uSu^{-1}\cdot S_{up}\subseteq \mathbb R^d$ given by $\phi_{up}(t,s)= u\exp(t)u^{-1}\cdot s$, for all $t\in T$ and $s\in S_{up}$. In particular, $uSu^{-1}\cdot S_{up}$ is open in $\mathbb R^d$ since $\dim(T\times S_{up}) = \dim(uSu^{-1}\cdot S_{up}) = d$.

    Now, it holds that $\{0\}\times S_{up} \subseteq \phi^{-1}_{up}(uSu^{-1}\cdot S_{up}\cap S\cdot V_x)$ and so by continuity, there exists an open neighborhood $T_{up}$ of $0$ in $T$ such that $T_{up}\times S_{up} \subseteq \phi^{-1}_{up}(uSu^{-1}\cdot S_{up}\cap S\cdot V_x)$. This proves (b). Additionally, $\Phi_{up}:=\phi_{up}|_{T_{up}\times S_{up}}$ is the diffeomorphism desired in (c).

    \vspace{\claimstepvspace}\noindent\textit{Step 3. We prove (a).}

    We aim to show that $\pi|_{S_{up}}$ is injective. Let $q_1,q_2\in S_{up} \subseteq S\cdot V_x\cap V_{up}$ and suppose that $\pi(q_1)=\pi(q_2)$. Then there exists $p'\in V_x$ such that $q_1,q_2\in Sp'$ and so $p'\in S^{-1}q_j = Sq_j$. Now, $u^{-1}q_1,u^{-1}q_2\in V_p$ and $p'\in S^2u^{-1}q_1\cap S^2u^{-1}q_2$. By \cref{claim.fiber disjoint}, we obtain that $u^{-1}q_1=u^{-1}q_2$ and so $q_1=q_2$ as desired.
    
    \vspace{\claimstepvspace}\noindent\textit{Step 4. We prove (d).}

    In order to show that the diagram in \eqref{diag.nice Phi} commutes, we show that $\pi(\Phi_{up}(t,q))=\pi(q)$ for each $t\in T_{up}$ and $q\in S_{up}\subseteq V_{up}\cap S\cdot V_x$. Let $p'\in V_x$ be such that $q \in Sp'$. Then $\Phi_{up}(t,q)\in S^3q\cap S\cdot V_x \subseteq S^4p'\cap S\cdot V_x$. By \cref{claim.fiber disjoint}, it follows that $\Phi_{up}(t,q)\in Sp'$ and so $\pi(\Phi_{up}(t,q))=\pi(q)=\pi(p')$ as desired.
\end{proof}

    We give a first application of these nice chart constructions.
    \begin{claim}[Topological Manifold]
        \label{claim.topological manifold}
    $M$ is a topological manifold. In fact, a cover by Euclidean charts is given by the collection $\{\pi|_{S_{up}}\}_{(u,p)\in S\times V_x}$, where $S_{up}$ is defined in \cref{claim.nice charts SVX}.
    \end{claim}
    \begin{proof}
    We have already shown that $M$ is Hausdorff and second countable in \cref{claim.Hausdorff}. It is left show that $\{\pi|_{S_{up}}\}_{(u,p)\in S\times V_x}$ do indeed form a cover by Euclidean charts.
    
    To this end, let $u\in S$ and $p\in V_x$ be arbitrary and put $W:= \pi(S_{up})$. Then $\pi|_{S_{up}}\colon S_{up}\to W$ is continuous and bijective by \cref{claim.pi sup injective}.
    
    It is left to show that $\pi|_{S_{up}}$ is an open map. To this end, let $T_{up}$ and $\Phi_{up}$ be as defined in \cref{claim.nice charts SVX}, and let $Y$ be an open subset of $S_{up}$. Then by commutativity of the diagram in \eqref{diag.nice Phi}, it holds that $\pi|_{S_{up}}(Y) = \pi\circ \Phi_{up}(T_{up}\times Y)$. The latter set is open since $\pi \circ \Phi_{up}$ is an open map and $T_{up}\times Y$ is an open subset of $T_{up}\times S_{up}$.
    \end{proof}
    
    Next, we give $M$ a smooth structure so that $\pi$ becomes a smooth submersion. For the proof of the following claim, see \cref{fig:desingfigure} for an illustration.
    \begin{claim}[Smooth Structure]
        \label{claim.smooth structure}
        The charts $\{\pi|_{S_{up}}\}_{(u,p)\in S\times V_x}$ generate a smooth structure on $M$ which turns $\pi$ into a smooth submersion.
    \end{claim}
    \begin{proof}

    \noindent\textit{Step 1. We deduce the submersion claim from smooth transitions.}

        Observe that the diagram in \eqref{diag.nice Phi} is a commutative diagram of Euclidean charts that turn $\pi$ into $\Pi_{S_{up}}\colon T_{up}\times S_{up}\to S_{up}$, the projection onto the component of $S_{up}$. Due to this and once we show that the charts $\pi|_{S_{up}}$ transition smoothly, we obtain that $\pi$ is a smooth submersion as desired.
    
        \vspace{\claimstepvspace}\noindent\textit{Step 2. We align two nice charts on $S\cdot V_x$ given in \cref{claim.nice charts SVX} and setup notation.}

        We are left to show that the transitions between charts $\pi|_{S_{up}}$ are smooth. Let $\pi|_{S_1}\colon S_1\to W_1$ and $\pi|_{S_2}\colon S_2\to W_2$ be two arbitrary charts, where $S_j := S_{u_jp_j}$ for some $u_j\in S$ and $p_j\in V_x$. Similarly, put $T_{j}:=T_{u_jp_j}$ and $\Phi_j := \Phi_{u_jp_j}$ as given by \cref{claim.nice charts SVX}. Suppose that $W_1$ and $W_2$ overlap, i.e., $\pi|_{S_1}(q_1) = \pi|_{S_2}(q_2) \in W_1\cap W_2$ for some $q_j\in S_j$. Then there exists $u'\in S^2$ such that $q_2 = u'q_1$. Now, the map $u'\cdot \Phi_1\colon T_1\times S_1 \to u'\cdot u_1\exp(T_1)u_1^{-1}\cdot S_1$ is a diffeomorphism, and we have that $q_2 \in \operatorname{im}(u'\cdot \Phi_1)\cap \operatorname{im}(\Phi_2)$.
                
        Put $K_1 := (u'\cdot \Phi_1)^{-1}(\operatorname{im}(u'\cdot \Phi_1)\cap \operatorname{im}(\Phi_2))$, $K_2:= \Phi_2^{-1}(\operatorname{im}(u'\cdot \Phi_1)\cap \operatorname{im}(\Phi_2))$ and $B_j:= \pi|_{S_j}^{-1}(W_1\cap W_2)$. Moreover, let $\Pi_j\colon T_j\times S_j\to S_j$ denote projection onto the component of $S_j$. Our main goal is to show that $\pi|_{S_1}^{-1}\circ \pi|_{S_2}$ is a diffeomorphism near $q_2\in B_2$. Then the result folows since $q_2\in \pi|_{S_2}^{-1}(W_1\cap W_2)$ was chosen arbitrarily.\vspace{\claimstepvspace}

\begin{samepage}
        \noindent\textit{Step 3. We show that the following diagram is well-defined and that its bottom square commutes.}
        \[
        \begin{tikzcd}
        & \operatorname{im}(u'\cdot \Phi_1)\cap \operatorname{im}(\Phi_2) \arrow[rrr, "\pi"] & & & W_1\cap W_2\\
        K_1 \arrow[rrrrr, bend right = 12, shift right = 1.2, "\Pi_1", swap]\arrow[ur, "u'\cdot \Phi_1|_{K_1}"]  & & K_2 \arrow[ll, "(u'\cdot \Phi_1|_{K_1})^{-1}\circ \Phi_2|_{K_2}",swap]\arrow[r,bend right, "\Pi_2"] \arrow[ul,"\Phi_2|_{K_2}",swap] & B_2\arrow[rr, "\pi|_{S_1}^{-1}\circ \pi|_{S_2}"]\arrow[ur,"\pi|_{S_2}"] & & B_1\arrow[ul,"\pi|_{S_1}",swap]
        \end{tikzcd}
        \]
\end{samepage}
        In other words, we will verify that $\pi(\operatorname{im}(u'\cdot \Phi_1)\cap \operatorname{im}(\Phi_2)) \subseteq W_1\cap W_2$, $\Pi_2(K_2)\subseteq B_2$, $\Pi_1(K_1)\subseteq B_1$ and that the bottom square commutes. We begin with showing that the outermost square commutes, i.e., $\pi\circ u'\cdot \Phi_1|_{K_1} = \pi|_{S_1}\circ \Pi_1|_{K_1}$ (in particular, this entails that $\pi(\operatorname{im}(u'\cdot \Phi_1)\cap \operatorname{im}(\Phi_2)) = \pi|_{S_1}(\Pi_1(K_1))\subseteq W_1$.) By the commutativity of the diagram in \eqref{diag.nice Phi}, this is equivalent to $\pi\circ u'\cdot \Phi_1|_{K_1} = \pi\circ \Phi_1|_{K_1}$. 
        To establish this, take any $(t,q)\in K_1\subseteq T_1\times S_1$ and let $p'\in V_x$ be such that $\Phi_1|_{K_1}(t,q)\in Sp'$. Then $u'\cdot \Phi_1|_{K_1}(t,q) \in S^3p'$ and so \cref{claim.fiber disjoint} entails that $\pi\circ u'\cdot \Phi_1|_{K_1}(t,q) = \pi\circ \Phi_1|_{K_1}(t,q) = \pi(p')$ as desired.
        
        Next, note that the middle square commutes by $\eqref{diag.nice Phi}$, i.e., $\pi|_{S_2}\circ \Pi_2|_{K_2}  = \pi\circ \Phi_2|_{K_2}$. In particular, this entails that $\pi(\operatorname{im}(u'\cdot \Phi_1)\cap \operatorname{im}(\Phi_2)) = \pi|_{S_2}(\Pi_2(K_2)) \subseteq W_2$. As such,
        \[\pi(\operatorname{im}(u'\cdot \Phi_1)\cap \operatorname{im}(\Phi_2)) = \pi|_{S_1}(\Pi_1(K_1)) = \pi|_{S_2}(\Pi_2(K_2)) \subseteq W_1\cap W_2,\] 
        and so $\Pi_j(K_j)\subseteq B_j$ for each $j\in \{1,2\}$ as we wished to verify. In particular, the bottom square is well-defined and we are now ready to show that it commutes by the following chain of equalities:
        \begin{equation}
            \label{eq.bottom square}
            \pi|_{S_1}^{-1}\circ \pi|_{S_2} \circ \Pi_2|_{K_2} =
            \pi|_{S_1}^{-1}\circ \pi \circ \Phi_2|_{K_2} = \Pi_1|_{K_1}\circ (u'\cdot \Phi_1|_{K_1})^{-1} \circ \Phi_2|_{K_2},
        \end{equation}
        where the first and second equalities hold since the middle and the outermost squares commute, respectively.

        \vspace{\claimstepvspace}\noindent\textit{Step 4. We finish the proof by showing that $\pi|_{S_1}^{-1}\circ \pi|_{S_2}$ is a diffeomorphism near $q_2$.}

        Since $(u'\cdot \Phi_1|_{K_1})^{-1} \circ \Phi_2|_{K_2}$ is a diffeomorphism, it has the form $(A(t,q),B(t,q))$ for $(t,q)\in K_2$. By \eqref{eq.bottom square}, we get that $\pi|_{S_1}^{-1}\circ \pi|_{S_2}(q) = B(t,q)$ for each $(t,q)\in K_2$. In particular, $B(t,q) = \tilde B(q)$, where $\tilde B := (\pi|_{S_1}^{-1}\circ \pi|_{S_2})|_{\Pi_2(K_2)}$ is now a smooth homeomorphism. Moreover, $\tilde B$ is a submersion since $(A,\tilde B)$ is a submersion. It follows that $(\pi|_{S_1}^{-1}\circ \pi|_{S_2})|_{\Pi_2(K_2)}$ is a diffeomorphism. Lastly, note that $q_2 \in \Pi_2(K_2)$ since $q_2 \in \operatorname{im}(u'\cdot \Phi_1)\cap \operatorname{im}(\Phi_2)$ and so $\Pi_2(\Phi_2|_{K_2}^{-1}(q_2)) = \pi|_{S_2}^{-1}(\pi(q_2)) = q_2$ as desired.
    \end{proof}

    Next, we complete the proof for the statement of \cref{lem.desingularization Mx smooth structure}.

    \begin{claim}[Riemannian Structure]
        \label{claim.riemannian submersion}
        Equip $M$ with the symmetric bilinear form $\gamma$ defined by
        \begin{equation}
            \label{eq.gamma in claim}
        \gamma_{\pi(up)}(d\pi(up)|_{uX_p}, d\pi(up)|_{uY_{p}}) := \langle uX_{p}, uY_{p}\rangle = \langle X_p,Y_p\rangle,\end{equation}
        for $X_{p},Y_{p}\in N_{p}$, $p\in V_x$ and $u\in S$. Then $\gamma$ is a well-defined (independent of $u$) smooth Riemannian metric and it turns $\pi\colon(S\cdot V_x,\langle{,}\rangle)\to(M,\gamma)$ into a smooth Riemannian submersion whose horizontal subspaces are given by
        \[H_{up} := \ker(d\pi(up))^\perp = N_{up},\]
        for $u\in S$ and $p\in V_x$.
    \end{claim}
    \begin{proof}Fix arbitrary $u\in S$ and $p\in V_x$, and let $\pi|_{S_{up}}$ and $\Phi_{up}$ be charts as defined in \cref{claim.nice charts SVX}.

        \vspace{\claimstepvspace}\noindent \textit{Step 1. We describe an orthogonal decomposition of the tangent space $T_{up}(S\cdot V_x)$.}

        The linear map $\mathcal T_{up}\colon \mathfrak g\to \mathbb R^d$ given by $\mathcal T_{up}(\omega) = u\cdot \omega \cdot p$ has kernel $T_eG_x^0$ and image $\mathfrak g \cdot up$. By the first isomorphism theorem and since $S$ intersects $G_x^0$ transversally at $e$, we get that $\mathcal T_{up}|_{T_eS}\colon T_eS\to \mathfrak g \cdot up$ is a vector space isomorphism. As such, an arbitrary tangent vector to $up$ has an orthogonal decomposition $\mathcal T_{up}(\omega) \oplus uX_p \in (\mathfrak g\cdot up)\oplus N_{up}$ for some unique $X_p\in N_p$ and $\omega \in T_eS$. 
        
        \vspace{\claimstepvspace}\noindent \textit{Step 2. We show that $d\pi(up)|_{\mathcal T_{up}(\omega) \oplus uX_p} = d\pi|_{S_{up}}(up)|_{uX_p}$ and so $H_{up}:=\ker(d\pi(up))^\perp = N_{up}$.}

        By the commutative diagram in \eqref{diag.nice Phi}, it holds that 
        \[d\pi(up)|_{\mathcal T_{up}(\omega) \oplus uX_p} = d\pi|_{S_{up}}\circ d\Pi_{S_{up}}\circ d\Phi_{up}^{-1}(up)|_{\mathcal T_{up}(\omega) \oplus uX_p}.\]
    Taking a curve $(t\omega,up+tuX_p)\in T_{up}\times S_{up}$, we use the product rule to obtain that
        \[d\Phi_{up}(0,up)|_{(\omega,uX_p)} = \left.\frac{d(u\exp(t\omega)u^{-1}(up+tuX_p))}{dt}\right|_{t=0} = \mathcal T_{up}(\omega) + uX_p.\]
        As such,
        \[d\Phi_{up}^{-1}(up)|_{\mathcal T_{up}(\omega)  \oplus uX_p} = (\omega,uX_p) \in T_eS\times N_{up}.\]
        Next, $d\Pi(0,up)|_{(\omega, uX_p)} = uX_p \in N_{up}$. As such, $d\pi(up)|_{\mathcal T_{up}(\omega) \oplus uX_p} = d\pi|_{S_{up}}(up)|_{uX_p}$ and so $\ker(d\pi(up)) = \mathfrak g\cdot up$. This proves the assertion that $\ker(d\pi(up))^\perp = N_{up}$.
        
        \vspace{\claimstepvspace}\noindent \textit{Step 3. We prove the isometry invariance $d\pi(up)|_{uX_p} = d\pi(p)|_{X_p}$ for every $X_p\in N_p$.}

        Consider the curve $t\mapsto up + tuX_p$ where $|t| <\varepsilon$ and $\varepsilon>0$ is small enough so that $up + tuX_p\in S_{up}$ and $p + tX_p\in S_p$ for each $t$. Then $\pi(up+ tuX_p) = \pi(p+ tX_p)$ for each $t$. To see this, set $q:= p+tX_p$ and let $p'\in V_x$ be such that $q\in Sp'$. Then $uq\in S^2p'\cap S\cdot V_x$ and so \cref{claim.fiber disjoint} entails that $\pi(q)=\pi(uq)=\pi(p')$ as desired. Taking derivatives at $t=0$, we arrive to the desired equality $d\pi(up)|_{uX_p} = d\pi(p)|_{X_p}$.
        
        \vspace{\claimstepvspace}\noindent \textit{Step 4. Finishing the proof.}

        By Step~3, we obtain that \eqref{eq.gamma in claim} is independent of $u\in S$, i.e., $\gamma$ is well-defined. By Step~2 and \eqref{eq.gamma in claim} and since $\pi$ is a smooth submersion, we obtain that $d\pi(up)|_{H_{up}}\colon H_{up} \to T_{\pi(up)}M$ is a vector space isomorphism that pushes forward the inner product $\langle{,}\rangle$ over $H_{up}$ into $\gamma_{\pi(up)}$. This entails that $(T_{\pi(up)}M,\gamma_{\pi(up)})$ is a Hilbert space and $d\pi(up)|_{H_{up}}\colon (H_{up},\langle{,}\rangle) \to (T_{\pi(up)}M,\gamma_{\pi(up)})$ is a Hilbert space isometry. Once we show that $\gamma$ smooth, we obtain that $\gamma$ is a Riemannian metric for $M$ that turns $\pi$ into a Riemannian submersion as desired.

        We are left to show that $\gamma$ is smooth. Set $n:= \dim(T_eS)$ and let $\omega_1,\dots, \omega_{n}$ be a basis of $T_eS$. Let $p\in V_x$ be arbitrary. Then the diffeomorphism $\Phi_p\colon T_{p}\times S_p\to \exp(T_p)\cdot S_p$ entails that $\{\omega_i\cdot q\}$ is a basis of $\mathfrak g\cdot q$ for each $q\in S_p$. This basis varies smoothly in $q$ and so we obtain smoothness of the map $P\colon S_p\times \mathbb R^d \to \mathbb R^d$ where $P_qx := P(q,x)$ sends $x\in\mathbb R^d$ to its orthogonal projection onto $N_q$. This is due to smoothness of the Gram-Schmidt process over a smoothly varying basis. Now, put $W:= \pi(S_p)$ and observe that $\sigma:= \pi|_{S_p}^{-1}\colon W\to S_p$ is a smooth chart that is also a local section of $\pi$, i.e., $\pi\circ \sigma = \operatorname{id}_W$. As such, we get that
        \[\gamma_{w}(v_1,v_2) = \gamma_{\pi(\sigma(w))}(d\pi\circ d\sigma(w)|_{v_1}, d\pi \circ d\sigma(w)|_{v_2}) = \langle P_{\sigma(w)}\circ d\sigma(w)|_{v_1}, P_{\sigma(w)}\circ d\sigma(w)|_{v_2}\rangle,\]
        which proves that $\gamma_{w}(v_1,v_2)$ is smooth in $w\in W$ and $v_1,v_2\in T_wM$ as desired.
    \end{proof}

    \subsubsection{Proof of Lemma~\ref{lem.desingularization embedded fiber}}
    Recall that fibers of smooth submersions are embedded submanifolds (e.g. Corollary~5.13 in~\cite{Lee:13}.) For each $p\in V_x$ and since $\pi$ is a smooth (Riemannian) submersion by \cref{claim.riemannian submersion}, the fiber $S\cdot p$ of $\pi$ is an embedded submanifold of $S\cdot V_x$ and hence of $\mathbb R^d$ and $[p]$ since $S\cdot V_x$ is an open subset of $\mathbb R^d$ as we have shown in \cref{claim.openness}. Lemma~\ref{lem.desingularization embedded fiber} now follows by noting that $\dim(S\cdot p) = \dim(N_p^\perp) = \dim([p])$.
    \subsubsection{Proof of Lemma~\ref{lem.desingularization WP}}

    The last two claims are geared to proving the statement of \cref{lem.desingularization WP}, after which we are done. For the proof of the following claim, see \cref{fig:claimWp} for an illustration.

    \begin{figure}[t]
        \centering
        \includegraphics[scale=1]{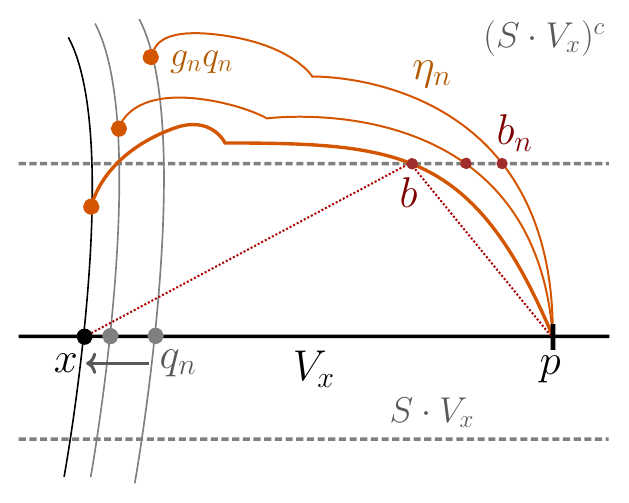}
        \caption{This figure is an aiding illustration for the proof of \cref{claim.Wp neighbhorhood}. While it is depicted that the curves $\eta_n$ converge pointwise to a limiting curve, this need not be the case.}
        \label{fig:claimWp}
      \end{figure}

    \begin{claim}[No Escape Neighborhood]
        \label{claim.Wp neighbhorhood}
        For each $p\in V_x$, there exists an open neighborhood $W_p$ of $x$ in $V_x$ such that $\pi(W_p)$ is open in $M$ and for each $q\in W_p$, $g\in G$ and piecewise unit speed $C^1$ curve $\eta\colon[0,L(\eta)]\to \mathbb R^d$ joining $p$ to $gq$ with length $L(\eta)\leq d(p,\overline Sq) = d(p,Sq)$, it holds that $\im(\eta)\subseteq S\cdot V_x$.
    \end{claim}
    \begin{proof}
        We establish the second assertion. This way, the first assertion follows from shrinking $W_p$ so that it lies in the neighborhood $S_x$ over which \cref{claim.topological manifold} applies and gives that $\pi(W_p)$ is open in $M_x$. To this end, suppose for the sake of contradiction that there are sequences $g_n\in G$, $q_n\in V_x$ with $q_n\to x$ and unit speed piecewise $C^1$ curves $\eta_n\colon [0,L(\eta_n)]\to \mathbb R^d$ joining $p$ to $g_nq_n$ with length $L(\eta_n)\leq d(p,\overline Sq_n)$ such that there exists a `witness of escape' $b_n \in \im(\eta_n)\cap (S\cdot V_x)^c$ for each $n$.
        
        Since $d(p,\overline Sq_n)\to d(p,\overline Sx)$ and $L(\eta_n)\leq d(p,\overline Sq_n)$, we may take subsequences so that $L(\eta_n)\rightarrow L$ and we have that
        \begin{equation}
            \label{eq.claim L px}
            L \leq d(p, \overline Sx) \leq d(p,Gx) = \|p-x\|.
        \end{equation}
        We may also take subsequences so that $g_n\to g$. Then since we have that $\|p-g_nq_n\|\leq L(\eta_n)$, we take limits to obtain that $\|p-gx\| \leq \|p-x\|$. Thus, $gx=x$ as $p \in V_x$.

        Next, $\|b_n-p\|\leq L(\eta_n)$ is bounded and hence we may assume $b_n\to b$ after taking subsequences. Then  $b\in (S\cdot V_x)^c$ since $b_n \in (S\cdot V_x)^c$ and $S\cdot V_x$ is open by \cref{claim.openness}. Since $[p,x]\subseteq S\cdot V_x$, it follows that $b\notin [p,x]$ so that the strictness of the triangle inequality yields
        \begin{equation}
        \label{eq.claim 1 over M}
        \|p-x\| < \|p-b\| + \|b-x\| - \frac{1}{M},
        \end{equation}
        for some large $M>0$. Then by \eqref{eq.claim L px} and \eqref{eq.claim 1 over M} and for large $n$, it holds that
        \[L(\eta_n) < L + \frac{1}{2M} \leq \|p-x\| + \frac{1}{2M}< \|p-b\| + \|b-x\| - \frac{1}{2M}.\]
        With this and since $b_n\in \im(\eta_n)$, the following chain of inequalities holds for large $n$
        \[\|p-b_n\| + \|b_n-g_nq_n\|\leq L(\eta_n) < \|p-b\| + \|b-x\| - \frac{1}{2M}.\]
        This contradicts the limits $b_n\to b$ and $g_nq_n \to gx=x$.
    \end{proof}
     
    We are alas ready to finish the proof with the following claim; see \cref{fig:liftpieces} for a proof aid illustration.

    \begin{claim}[Minimal Geodesics]
        \label{claim.minimallifts}
        Suppose that $p\in V_x$ and $q\in W_p$, where $W_p$ is given in \cref{claim.Wp neighbhorhood}. Then, minimal geodesics from $\pi(p)$ to $\pi(q)$ in $M$ exist and are precisely the $\pi$-images of straight line distance minimizers from $p$ to $Sq$, all of which lie in $S\cdot V_x$.
    \end{claim}
    \begin{proof}
        By \cref{claim.Wp neighbhorhood}, every distance minimizing straight line $[p,uq]$ joining $p$ to $\overline Sq$ lies in $S\cdot V_x$ and is as such a straight line distance minimizer from $p$ to $Sq$ since $\overline Sq\cap S\cdot V_x = Sq$ by \cref{claim.fiber disjoint}. Since $\pi$ is a Riemannian submersion and $p-uq\in N_{uq}$ is a horizontal direction, we obtain that $\pi([p,uq])$ is a geodesic in $M$ joining $\pi(p)$ to $\pi(q)$ and of length $d(p,Sq)$ (see Proposition~2.109 in~\cite{GallotHL:90}.)
        
        Now, let $c_\gamma:[0,L(c_\gamma)]\to M$ be a unit speed piecewise $C^1$ curve joining $\pi(p)$ to $\pi(q)$ where $L(c_\gamma)$ denotes the length of $c_\gamma$, and suppose that $L(c_\gamma)\leq d(p,Sq)$. The claim follows once we show that $L(c_\gamma)= d(p,Sq)$ and that $c_\gamma$ is the $\pi$-image of a straight line distance minimizer from $p$ to $Sq$ in $S\cdot V_x$. The rest of the proof is dedicated to this objective.
        
        \vspace{\claimstepvspace}\noindent \textit{Step 1. We reduce to $c_\gamma$ being a piecewise geodesic curve.}
        
        By covering $\im(c_\gamma)$ with finitely many geodesically convex neighborhoods $\{B_i\}_{i=1}^N$ and taking a partition $0=T_0 < T_1 < \cdots < T_N = L(c_\gamma)$ that satisfies $c_\gamma([T_{i-1},T_i])\subseteq B_i$ for each $i$, we may redraw each $c_\gamma|_{[T_{i-1},T_i]}$ as the unique minimal geodesic joining $c_\gamma(T_i)$ to $c_\gamma(T_{i+1})$. If this results in $L(c_\gamma) < d(p,Sq)$, then the mere proof of $L(c_\gamma)= d(p,Sq)$ yields a contradiction; otherwise, the redrawing keeps $c_\gamma$ unaltered by the unique minimality of the pieced geodesics in each $B_i$. As such, we may now assume without loss of generality that $c_\gamma$ is a piecewise geodesic with each geodesic piece given by $c_\gamma|_{[T_{i-1},T_i]}$.

        \vspace{\claimstepvspace}\noindent \textit{Step 2. We horizontally lift $c_\gamma$ piece by piece.}

        By local horizontal lifting of geodesics under Riemannian submersions (Proposition~2.109 in~\cite{GallotHL:90},) there exists a refined partition $0=t_0 < t_1 < \cdots < t_n = L(c_\gamma)$ and for each $i\in \{0,\dots, n-1\}$, there exists $p_i \in \pi^{-1}(c_\gamma(t_i))\cap V_x$ and $q_{i+1}\in \pi^{-1}(c_\gamma(t_{i+1}))$ such that $[p_i,q_{i+1}]\subseteq N_{p_i}$ are straight lines connecting $p_i$ to $Sp_{i+1}$, entirely lying in $S\cdot V_x$, and horizontally lifting $c_\gamma|_{[t_i,t_{i+1}]}$ to $S\cdot V_x$ starting from $p_i$, i.e., $\pi([p_i,q_{i+1}])= \im(c_\gamma|_{[t_i,t_{i+1}]})$. We take $p_0 := p$ and $p_n := q$. In particular, observe that $\sum_{i=0}^{n-1} \|q_{i+1}-p_i\| = L(c_\gamma)$.
    
        \begin{figure}[t]
            \centering
            \includegraphics[scale=1.3]{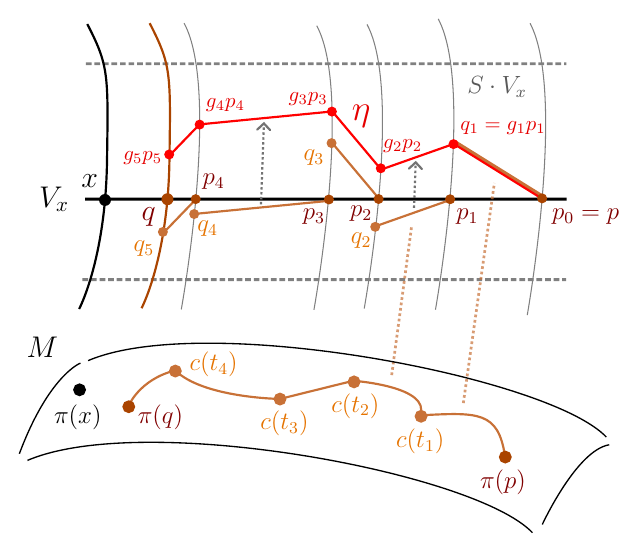}
            \caption{This figure is an aiding illustration for the proof of \cref{claim.minimallifts}, where $n=5$. It represents the situation after we conclude that $\im(\eta)\subseteq S\cdot V_x$ in Step~4 but before we conclude that $\eta$ is a straight line and $c:= c_\gamma$ is a minimal geodesic in Step~5. The dotted orange lines connecting portions of $c$ in $M$ to portions of $\eta$ in $S\cdot V_x$ are solely for illustrating the process of horizontal lifting. The gray arrows joining $[p_i,q_{i+1}]$ to $[g_ip_i, g_{i+1}p_{i+1}]$ depict the notion of making minimal moves to stitch $\eta$ together.}
            \label{fig:liftpieces}
          \end{figure}    

        \vspace{\claimstepvspace}\noindent \textit{Step 3. We take minimal moves in $G$ to attach the lifts together into a piecewise $C^1$ curve.} 

        For each $i\in \{0,\dots, n-1\}$, there exists $u_{i+1}\in S$ such that $u_{i+1}p_{i+1} = q_{i+1}$. Take $u_0:=e$ and equip $G$ with a bi-invariant Riemannian metric $\beta$. We inductively construct a sequence of `moves' $\{g_i\}_{i=0}^n\in G$. First, set $g_0:=e$. Assuming that we have defined $g_i$ for some $i \in \{0,\dots, n-1\}$, we proceed to define $g_{i+1}$.
        
        Put $F_i = \{g\in G: gp_{i+1}=g_iq_{i+1}\}\neq\varnothing$. Then $F_i$ is a compact subset of $G$ and its nonemptiness is witnessed by $g_iu_{i+1}\in F_i$. With this, we define $g_{i+1}\in G$ by fixing a minimal choice 
        \begin{equation}
            \label{eq.minimalchoiceg}
            g_{i+1} \in 
            \arg\min_{g\in F_i}d_\beta(g, S).
        \end{equation}
        Then $g_{i+1}\in F_i$, i.e., $g_{i+1}p_{i+1}=g_iq_{i+1}$, for each $i\in \{0,\dots, n-1\}$, and the sequential concatenation of the straight line segments $g_i[p_i,q_{i+1}] = [g_ip_i,g_{i+1}p_{i+1}]$, for $i\in \{0,\dots, n-1\}$, yields a piecewise unit speed $C^1$ curve $\eta\colon [0,L(\eta)]\to \mathbb R^d$ joining $p$ to $g_nq$ with $L(\eta) = L(c_\gamma)$. 
        
        \vspace{\claimstepvspace}\noindent \textit{Step 4. We invoke \cref{claim.Wp neighbhorhood} and show that the minimal moves are in $S$.} 
        
        By construction of $\eta$, it holds that
        \begin{equation}
            \label{claim eq.Length eta bounds}
            \|p-g_nq\| \leq L(\eta) = L(c_\gamma)\leq d(p,Sq),
        \end{equation}
        and so $\im(\eta)\subseteq S\cdot V_x$ by \cref{claim.Wp neighbhorhood}. In particular, $g_iq_{i+1} = g_{i+1}p_{i+1}\in S\cdot V_x$ for each $i\in \{0,\dots, n-1\}$.
        
        From this, we shall find that $g_i\in S$ for each $i\in \{0,\dots, n\}$. We proceed by induction. The base case holds since $g_0=e\in S$. Next, suppose that $g_{i_0}\in S$ for some $i_0 \in \{0,\dots,n-1\}$. Then by definition of $u_{i_0+1}\in S$ and since $g_{i_0}q_{i_0+1} \in S\cdot V_x$, we have that
        \[g_{i_0}u_{i_0+1}p_{i_0+1} = g_{i_0}q_{i_0+1}\in S\cdot V_x.\]
        Since $g_{i_0}u_{i_0+1}\in S^2$, it follows that $g_{i_0}q_{i_0+1} \in S^2p_{i_0+1}\cap S\cdot V_x$. By \cref{claim.fiber disjoint}, we obtain that $g_{i_0}q_{i_0+1}\in Sp_{i_0+1}$, and so $g_{i_0}q_{i_0+1} = up_{i_0+1}$ for some $u\in S$. As such, we have $u\in F_i\cap S$, and so $d_\beta(F_i,S) = 0$. By \eqref{eq.minimalchoiceg}, we obtain that $g_{i_0+1}\in S$ as desired. In particular, it follows that $g_n\in S$.
        
        \vspace{\claimstepvspace}\noindent \textit{Step 5. We finish the proof.} 
        
        By \eqref{claim eq.Length eta bounds} and since $g_n\in S$, we get that
        \[d(p,Sq)\leq \|p-g_nq\| \leq L(\eta) = L(c_\gamma) \leq d(p,Sq).\]
        Hence, equalities hold and $\eta$ is precisely the straight line $[p,g_nq]$, which is a minimizer of distance from $p$ to $Sq$ lying entirely in $S\cdot V_x$. Lastly, $\pi\circ\eta = c_\gamma$ since each piece of $\eta$ is given by $g_i[p_i,q_{i+1}] \subseteq S\cdot V_x$ where $g_i\in S$ and $\pi(g_i[p_i,q_{i+1}])=\pi([p_i,q_{i+1}])= c_\gamma|_{[t_i,t_{i+1}]}$ (the first equality follows from \cref{claim.fiber disjoint} and the second equality follows by definition of $p_i$ and $q_{i+1}$.)
        \end{proof}

        \section{Semialgebraic geometric arguments}

This section is dedicated to proving \cref{lem.regular local avoidance}. We begin with a preliminary introduction to semialgebraic geometry in \cref{app.prelim semialg}. 
Notably, the statement of \textit{conservation of dimension} (\cref{prop.conservation of dimension}) is a key result and will be frequently referenced in \cref{app.regular local avoidance proof}, where we provide the proof of \cref{lem.regular local avoidance}.
\subsection{Preliminary on Semialgebraic Sets and Groups}
\label{app.prelim semialg}
A basic semialgebraic set is any set of the form $\{x\in\mathbb R^n:p(x)\geq0\}$, where $p\colon\mathbb{R}^n\to\mathbb{R}$ is a polynomial function.
A \textbf{semialgebraic set} is any set obtained from some combination of finite unions, finite intersections, and complements of basic semialgebraic sets. By Proposition~2.9.10 in~\cite{BochnakCR:13}, every semialgebraic set is a finite union of manifolds. As such, the \textbf{dimension} of a semialgebraic set is defined as the maximum dimension of said manifolds.
We say a subgroup of $\operatorname{GL}(d)$ is a \textbf{semialgebraic group} if it is semialgebraic as a subset of $\mathbb{R}^{d\times d}$.
We say a function $\mathbb{R}^s\to\mathbb{R}^t$ is a \textbf{semialgebraic function} if its graph is semialgebraic as a subset of $\mathbb{R}^{s+t}$.

A pivotal observation is that, starting with a fixed collection of finitely many semialgebraic sets, one can construct new semialgebraic sets through the application of first-order logic.

\begin{definition}
A first-order formula of the language of ordered fields
with parameters in $\mathbb R$ is a formula written with a finite number of conjunctions, disjunctions, negations, and universal or existential quantifiers
on variables, starting from atomic formulas which are formulas of the kind $f(x_1,\dots,x_n) = 0$ or $g(x_1,\dots,x_n) > 0$, where $f$ and $g$ are polynomials with
coefficients in $\mathbb R$. The free variables of a formula are those variables of the polynomials appearing in the formula, which are not quantified.
\end{definition}

\begin{proposition}[Proposition~2.2.4 in~\cite{BochnakCR:13}]\label{prop.firstorder semialgebraic}
    Let $\phi(x_1,\dots,x_n)$ be a first-order formula of the language of ordered fields, with parameters in $\mathbb R$ and with free variables $x_1,\dots,x_n$. Then $\{x \in \mathbb R^n: \phi(x)\}$ is a semialgebraic set.
\end{proposition}

The above principle allows one to reveal structure in the family of semialgebraic sets. The following proposition demonstrates aspects of this, and proofs of statements therein can be found in Appendix~A of~\cite{AmirGARD:23}.
\begin{proposition}
    The following statements regarding semialgebraic sets and functions hold:
    \begin{propenum}
        \item \label{prop.semialg set family} The family of semialgebraic sets is closed under coordinate projection, complement, finite union, finite intersection and coordinate slicing.
        \item \label{prop.semialg function family} The family of semialgebraic functions is closed under addition, multiplication, division (when defined), composition and concatenation. Moreover, fibers and images of semialgebraic functions are semialgebraic sets.
        \item \label{prop.conservation of dimension}(Conservation of Dimension) If $\pi\colon\mathbb R^{n+d}\mapsto \mathbb R^n$ is a coordinate projection and $A$ is a semialgebraic subset of $\mathbb R^{n+d}$, then
        \begin{equation}
            \dim(\pi(A)) \leq \dim(A) \leq \dim(\pi(A)) + \max_{x\in\pi(A)}\dim(\pi^{-1}(x)\cap A).
        \end{equation}
    \end{propenum}
\end{proposition}
As mentioned in the beginning of the section, \textit{conservation of dimension} is essential to many arguments in this paper. The next proposition highlights that semialgebraicity of a group is equivalent to its compactness.

\begin{proposition}[Proposition~7 in~\cite{MixonQ:22}]
    \label{prop.closed subgroups}
    Suppose $G\leq\operatorname{O}(d)$.
    The following are equivalent:
    \begin{itemize}
    \item[(a)]
    $G$ is topologically closed.
    \item[(b)]
    $G$ is algebraic.
    \item[(c)]
    $G$ is semialgebraic.
    \end{itemize}
\end{proposition}
As a first application of the above propositions, we prove that the collections of principal and regular points, as well as the components of the Voronoi decomposition, form semialgebraic sets. Additionally, we demonstrate that quotient metrics and max filtering maps are semialgebraic functions.
\begin{proposition}
    \label{prop.distance is semialgebraic}
    For any compact $G\leq \Oname(d)$, the quotient metric $d([\cdot],[\cdot])\colon\mathbb R^d\times \mathbb R^d \to \mathbb R$ and the max filtering map $\llangle[\cdot],[\cdot]\rrangle\colon \mathbb R^d\times \mathbb R^d \to \mathbb R$ are semialgebraic functions. Moreover, the sets $P(G)$, $R(G)$, $N_x$, $U_x$, $V_x$ and $Q_x$ are semialgebraic subsets of $\mathbb R^d$, for all $x\in\mathbb R^d$.
\end{proposition}
\begin{proof}
By \cref{prop.closed subgroups}, $G$ is a semialgebraic group. Then the graph of $d([\cdot],[\cdot])$ is expressed through first order logic in the following form
\[\big\{(x,z,r)\in(\mathbb R^d)^2\times \mathbb R: (\forall g\in G, r\leq \|gx-z\|) \wedge (\forall\varepsilon\in\mathbb R, \exists g\in G, \varepsilon > 0 \Rightarrow r+\varepsilon > \|gx - z\|)\big\},\]
to which \cref{prop.firstorder semialgebraic} applies. A similar argument applies to the max filtering map. To establish the rest of the proposition, first observe that $P(G)$ is expressed through first order logic in the following form
\[P(G)=\big\{x\in\mathbb R^d: \forall p\in \mathbb R^d, (\exists g\in G, gp=p \wedge gx\neq x) \vee (\forall g\in G, gp = p \iff gx=x)\big\}.\]
In other words, it is either the case that $G_p$ is not a subgroup $G_x$ or otherwise, $G_p = G_x$. Next, set $D:=\max_{x\in \mathbb R^d}\dim([x])$ and let $\omega_1,\dots, \omega_N$ be a basis of $\mathfrak g$, the Lie algebra of $G$. Then $R(G)$ is expressed with semialgebraic conditions in the following form
\[R(G)=\big\{x\in\mathbb R^d: \{\omega_i\cdot x\}_{i=1}^N \text{ contains $D$ linearly independent vectors}\big\}.\]
Next, for any $x\in \mathbb R^d$, $N_x$ is expressed with semialgebraic conditions in the following form
\[N_x = \big\{y\in\mathbb R^d: \langle y,\omega_i\cdot x\rangle = 0\ \forall\,  i\in\{1,\dots,N\}\big\}.\]
As for the Voronoi decomposition components, we define $U:=\{(x,y)\in(\mathbb R^d)^2: y\in U_x\}$ and we express it with semialgebraic conditions in the following form
\[
U = \big\{(x,y)\in (\mathbb R^d)^2: \|x-y\|=d([x],[y]) \wedge \forall\,  q\in G\cdot x-\{x\}, \|q-y\|\neq d([x],[y])\big\}.
\] 
By \cref{prop.semialg set family} and since $U_x$ is a coordinate slice of $U$, it follows that $U_x$ is a semialgebraic subset of $\mathbb R^d$.

Next, by \cref{lem. Q char}, $Q_x$ is expressed with semialgebraic conditions in the following form
\[Q_x=\big\{y\in\mathbb R^d: \exists \epsilon >0, \forall\, y'\in \mathbb R^d, |y'-y|<\epsilon \implies (\exists p\in G\cdot x,(p,y')\in U)\big\}.\]
Lastly, \cref{prop.semialg set family} applies to the finite intersection $V_x = Q_x\cap U_x$.
\end{proof}

\subsection{Proof of Lemma~\ref{lem.regular local avoidance}}
\label{app.regular local avoidance proof}

This section is dedicated to the proof of \cref{lem.regular local avoidance}, which is both long and technical. To enhance readability and organization, the proof is divided into subsections, each corresponding to a specific phase of the argument. The first phase is a sequence of reduction steps which force the `bad' templates $\{z_i\}$ to lie within a common open Voronoi cell $V_x$. The second phase demonstrates that such bad templates lead to a nontrivial kernel for one of finitely many linear operators formed by the templates. The third and final phase shows that such noninjectivity fails to hold for sufficiently many generic templates.

Each phase is divided into a sequence of claims, each accompanied by proof. We hope this structure enables the reader to follow the argument smoothly.

For $z_1,\dots,z_n\in \mathbb R^d$ and $x,y\in \mathbb R^d$ with $[x]\neq [y]$, we define
    \begin{equation}
        \label{eq.D defn}
        D(x,y) := \left\{\frac{\llangle[x],[z_i]\rrangle - 
    \llangle [y],[z_i]\rrangle}{d([x],[y])}\right\}_{i=1}^n.
    \end{equation}
    Notably, $D$ is dilation invariant and $G$-invariant, meaning that
    \[D(x,y) = D(rgx,rg'y), \text{ for all $r>0$ and $g,g'\in G$}.\] 
    
    \subsubsection{Reduction to templates lying within a fixed open Voronoi cell}
    We begin by reducing to a fixed but arbitrary witness of failure.
    \begin{claim}
        The set $R$ given by \eqref{eq.R def} is semialgebraic and to obtain the bound \eqref{eq.R bound}, it suffices to show that 
        \begin{equation}
            \label{eq.R1}
        \dim(R_1(x))\leq nd - 1 - \left(\left\lceil\frac{n}{\chi(G)}\right\rceil - c\right),
        \end{equation}
        for all $x\in R(G)$, where
        \[R_1(x) := \big\{\{z_i\}_{i=1}^n \in (\mathbb R^d)^n: \Phi \text{ fails to be locally lower Lipschitz at $x$}\big\}.\]
        Moreover, $R_1(x)$ is semialgebraic.
    \end{claim}
    \begin{proof}
    Fix any $p\in R(G)$ and put $N:= N_p$. By \cref{lem.Q open dense}, $Q_p=G\cdot V_p\subseteq G\cdot N$ is dense. Since $G\cdot N$ is closed, it follows that $G\cdot N = \mathbb R^d$ and hence $[N] = \mathbb R^d/G$, i.e., $N$ meets every orbit in at least one point.  Since $p\in R(G)$, observe that $\dim(N) = c$.

    Next, observe that $\Phi$ fails to be locally lower Lipschitz at $x\in R(G)$ if and only if there exist sequences $[x_j],[y_j] \to [x]$ such that $[x_j]\neq [y_j]$ and $D(x_j,y_j)\to 0$. Note that these conditions are dilation invariant and $G$-invariant. Then by setting $S:=R(G)\cap \mathbb S^{d-1}\cap N$, we have that $R$ is a coordinate projection of the following lift
    \[L := \big\{(\{z_i\}_{i=1}^n, x) \in (\mathbb R^d)^n \times S: \Phi \text{ fails to be locally lower Lipschitz at $x$}\big\}.\]
    By \cref{prop.semialg set family,prop.distance is semialgebraic}, $S$ is semialgebraic. By \cref{prop.firstorder semialgebraic}, $L$ is semialgebraic since it may be expressed in the language of first order logic as follows
    \[\begin{aligned}
        L &:= \big\{(\{z_i\}_{i=1}^n, x) \in (\mathbb R^d)^n \times S: \forall \varepsilon \in \mathbb R_{>0}, \exists x_0,y_0\in \mathbb R^d,\\
         & \qquad \qquad [x_0]\neq [y_0]\wedge [x_0],[y_0] \in B_{[x]}(\varepsilon) \wedge |D(x_0,y_0)| < \varepsilon \big\}.
     \end{aligned}
     \]
     By \cref{prop.semialg set family}, it follows that $R$ is semialgebraic since it is a coordinate projection of $L$. A similar first order expression shows that $R_1(x)$ is semialgebraic for each $x \in R(G)$. Moreover, by \cref{prop.conservation of dimension}, we have the bounds
     \[\dim(R)\leq \dim(L)\leq \dim(S) + \max_{x\in S}\dim(R_1(x)).\]
    Since $\dim(S) = c - 1$ and $S\subseteq R(G)$, the claim follows.
    \end{proof}

    Next, we reduce to templates lying within the open Voronoi diagram $Q_x$.
    \begin{claim}
    Fix arbitrary $x\in R(G)$. To obtain the bound \eqref{eq.R1}, it suffices to show that
    \begin{equation}
        \label{eq.R2}
        \dim(R_2) \leq md - 1 - \left(\left\lceil\frac{m}{\chi(G)}\right\rceil - c\right),
    \end{equation}
    for each $m\geq 0$, where
    \[R_2 := \big\{\{z_i\}_{i=1}^m \in Q_x^m: \Phi \text{ fails to be locally lower Lipschitz at $x$}\big\}.\]
    Moreover, $R_2$ is semialgebraic.
    \end{claim}
    \begin{proof}
    We decompose analysis based on which $z_i$ lies in $Q_x$. For $z_1,\dots,z_n\in \mathbb R^d$, define $I_{\{z_i\}}:=\{1_{z_i\in Q_x}\}_{i=1}^n\in \mathbb \{0,1\}^n$. Then by \cref{lem.Q semialgebraic}, $I$ is semialgebraic in $\{z_{i}\}_{i=1}^n$. We obtain a finite partition $R_1 = \bigsqcup_{I\in\{0,1\}^n} R_1^I$, where
    \[R_1^I:= R_1 \cap \big\{\{z_i\}_{i=1}^n \in (\mathbb R^d)^n: I_{\{z_i\}} = I\big\}.\]
    Then each $R_1^I$ is semialgebraic, and we have that $\dim(R_1)= \max_{I\in \{0,1\}^n}\dim(R_1^I)$. Without loss of generality, we may assume that the maximum is achieved by $I_* := \{1_{j\leq m}\}_{j=1}^n$ for some $m\in \{0,\dots, n\}$. Then $R_2$ is a coordinate projection of $R_1^{I_*}$, so it is semialgebraic by \cref{prop.semialg set family}. Additionally, since $R_1^{I_*}\subseteq R_2\times (Q_x^c)^{n-m}$, it follows that
    \[\dim(R_1) =\dim(R_1^{I_*}) \leq \dim(R_2) + \dim((Q_x^c)^{n-m}).\]
    Since $Q_x$ is open and dense in $\mathbb R^d$ (\cref{lem.Q open dense}), it holds that
    \[\dim((Q_x^c)^{n-m})\leq (n-m)(d-1).\]
     The claim follows by combining the above two inequalities and noting that
     \[m-n \leq \left\lceil\frac{m}{\chi(G)}\right\rceil - \left\lceil\frac{n}{\chi(G)}\right\rceil.\]
    \end{proof}

    Next, we reduce to the case of all templates lying in the open Voronoi cell $V_x$.
    \begin{claim}
        To obtain the bound \eqref{eq.R2}, it suffices to show that
        \begin{equation}
            \label{eq.R3}
            \dim(R_3) \leq mc - 1 - \left(\left\lceil\frac{m}{\chi(G)}\right\rceil - c\right),
        \end{equation}
        where
        \[R_3 := \big\{\{z_i\}_{i=1}^m \in V_x^m: \Phi \text{ fails to be locally lower Lipschitz at $x$}\big\}.\]
        Moreover, $R_3$ is semialgebraic.
        \end{claim}
    \begin{proof}    
     By \cref{lem.UV facts semialgebraic} and a first order logic expression, it holds that $R_3$ is semialgebraic. Since $Q_x = G\cdot V_x$, we have $R_2 = G^m\cdot R_3$. Then $R_2$ is a coordinate projection of $L_3\cap (\mathbb R^d\times R_3)$, where
     \[L_3 := \big\{(\{z_i\}_{i=1}^m,\{v_i\}_{i=1}^m)\in (\mathbb R^d)^m \times V_x^m: [z_i]= [v_i]\ \, \forall\, i \in \{1,\dots,m\}\big\}.\]
     By \cref{prop.conservation of dimension} and considering the coordinate projection $\pi_2\colon (\mathbb R^d)^m \times V_x^m\to V_x^m$, it follows that
     \[\dim(R_2)\leq \dim(L_3\cap (\mathbb R^d\times R_3)) \leq \dim(R_3) + \max_{\{v_i\}_{i=1}^m\in V_x^m}\dim(\pi_2^{-1}(\{v_i\}_{i=1}^m)).\]
     Since $\dim(\pi_2^{-1}(\{v_i\}_{i=1}^m))\leq m\cdot \max_{x\in\mathbb R^d}\dim(G\cdot x) = m(d-c)$, the claim follows.
    \end{proof}

    \subsubsection{Reduction to linear operator injectivity}
    The following core claim establishes that the bad templates in $R_3$ result in a failure of injectivity for one of finitely many linear operators determined by the templates. The proof is intricate and therefore divided into six steps for clarity. Furthermore, the proof relies on the geometric characterization of regular points provided in \cref{lem.regular char}.
    \begin{claim}
    \label{claim.R3E intersection}
    Define    
    \begin{equation}
        \label{eq.E definition}
        E:= \bigcap_{\{h_l\}_{l=1}^m\in (G_x)^m}
        \bigcap_{\substack{I\subseteq \{1,\dots,m\}\\
        |I|=\left\lceil\frac{m}{\chi(G)}\right\rceil}}\big\{ \{z_i\}_{i=1}^m\in N_x^m: \text{$\{\langle h_iz_i,\cdot\rangle|_{N_x}\}_{i\in I}$ is injective}\big\}.
    \end{equation}
        Then $R_3 \cap E = \varnothing$. 
    \end{claim}
    \begin{proof}
    \noindent \textit{Step 1. Establishing a concrete goal.}

    Suppose that there exists $\{z_i\}_{i=1}^m\in R_3\cap E$.
     By definition of $R_3$, there exist sequences $[x_j],[y_j]\to [x]$ such that $[x_j]\neq [y_j]$ and $\lim_{j\to \infty} D(x_j,y_j) = 0$. By $G$-invariance and since $x\in Q_x$, we may assume $x_j,y_j\in V_x\subseteq N_x$ and $x_j,y_j\to x$. Note that these assumptions remain $G_x$-invariant.
     
     To obtain a contradiction, it suffices to show that there exist $I\subseteq \{1,\dots,m\}$, with $|I|\geq \left\lceil\frac{m}{\chi(G)}\right\rceil$, and $u\in N_x$, with $\|u\| \geq 1$, such that after taking subsequences, the following convergence holds for each $i\in I$ and some $w_i\in [z_i]_{G_x}$:
    \begin{equation}
        \label{eq.convergence RG}
        \frac{\llangle[x_j],[z_i]\rrangle-\llangle[y_j],[z_i]\rrangle}{d([x_j],[y_j])} \to \langle w_i, u\rangle.
    \end{equation}
    This would contradict the injectivity of $\{\langle w_i,\cdot\rangle|_{N_x}\}_{i\in I}$, which is guaranteed by the definition of $E$. 

    \vspace{\claimstepvspace}\noindent \textit{Step 2. Finding $w_i$ as limits of nearest points.}

    To the end of establishing \eqref{eq.convergence RG}, define the set of nearest elements in $[z_i]$ to $y\in \mathbb R^d$ by
    \[N_{i}(y) := \{w\in [z_i]: \llangle[z_i],[y]\rrangle = \langle w, y \rangle\}.\]
    In particular, note that $N_i(x) = G_xz_i$ since $z_i\in V_x$. Now, take subsequences so that the limits $\lim_{j\to \infty}N_{i}(x_j)\subseteq G_xz_i$ and $\lim_{j\to \infty}N_{i}(y_j)\subseteq G_xz_i$ exist for each $i$. By the pigeonhole principle, there exists $h\in G_x$ and $I\subseteq \{1,\dots,m\}$ such that $|I|\geq \left\lceil\frac{m}{\chi(G)}\right\rceil$ and $\lim_{j\to \infty}N_{i}(x_j)\cap \lim_{j\to \infty}N_{i}(hy_j) \neq \varnothing$ for each $i\in I$. By $G_x$-invariance of the assumptions on the sequences $x_j$ and $y_j$, we adjust the sequence $(y_j)_{j\in\mathbb N}$ into $(hy_j)_{n\in\mathbb N}$ and define 
    \[L_i:=\lim_{j\to \infty}N_{i}(x_j)\cap \lim_{j\to \infty}N_{i}(y_j) \neq \varnothing,\]
    for each $i\in I$. Take any $w_i\in L_i\subseteq G_xz_i$ for each $i\in I$. We establish \eqref{eq.convergence RG} with these $w_i$.

    \begin{figure}[t]
        \centering
        \includegraphics[scale=1.1]{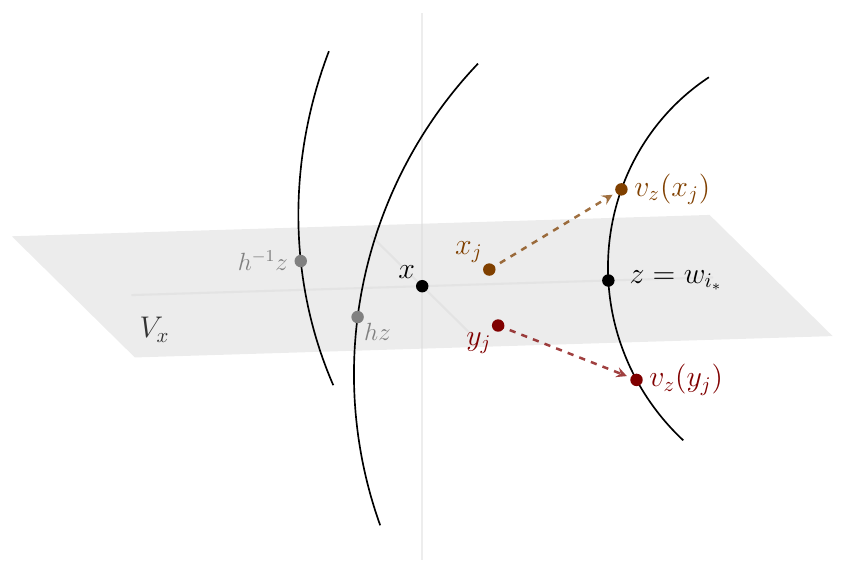}
        \caption{This figure is an aiding illustration for Steps~3~and~4 in the proof of \cref{claim.R3E intersection}. In this plot, we have $G_xz_{i_*} = G_xw_{i_*} = \{z,hz,h^{-1}z\}$. Note that in Step~2, we had already redefined $y_j$ so that there are nearest neighbors of $x_j$ and $y_j$ to $[z_{i_*}]$ (namely, $v_z(x_j)$ and $v_z(y_j)$, respectively) that both lie in a small neighborhood of $z=w_{i_*}$. This allows for analysis to go through in Steps~3~and~4, which this figure aims to visually aid in.} 
        \label{fig:linearanalysis}
      \end{figure}    

    \vspace{\claimstepvspace}\noindent \textit{Step 3. Linearized analysis of the max filter (see \cref{fig:linearanalysis}).}

    Fix arbitrary $i_* \in I$ and put $z := w_{i_*}$. By the definition of $L_{i_*}$, there exist sequences $z_j^x,z_j^y\to z$, where $z^x_j \in N_{i_*}(x_j)$ and $z^y_j \in N_{i_*}(y_j)$ for each $j$.

    Since $z\in G_xz_{i_*} \subseteq V_x$ and $x\in R(G)$, \cref{lem.regular char} entails that $x\in V_z^{loc}$. By definition of $V_z^{loc}$, there exists open neighborhoods $V$ of $x$ and $U$ of $z$ such that
    \[\forall q\in V, \ \ \Big|\arg\sup_{p\in[z]\cap U}\langle p,q\rangle \Big|= 1.\]
    By taking subsequences, we assume that $x_j,y_j\in V$ and $z^x_j,z^y_j \in U$ for each $j$. Furthermore, let $v_z\colon V \to [z]$ be the map such that $v_z(q)$ is the unique element in $\arg\sup_{p\in[z]\cap U}\langle p,q\rangle$. In particular, since $z\in V_x$, we have $v_z(x)=z$.
    
    Note that $z^x_j = v_z(x_j)$, $z^y_j = v_z(y_j)$, and $\|x_j-y_j\|\neq 0$ since $[x_j]\neq [y_j]$ for each $j$. Thus,
    \begin{equation}
        \label{eq.psi multiply divide helper}
        \frac{\llangle[x_j],[z]\rrangle-\llangle[y_j],[z]\rrangle}{d([x_j],[y_j])} = \frac{\|x_j-y_j\|}{d([x_j],[y_j])}\cdot \frac{\langle x_j, v_z(x_j)\rangle-\langle y_j,v_z(y_j)\rangle}{\|x_j-y_j\|}.
        \end{equation}
        In the following step, we work on the right-hand fraction in the product and in the step after, we bound the left-hand fraction in the product.

        \vspace{\claimstepvspace}\noindent \textit{Step 4. Analysis of the right-hand fraction in the right-hand side of \eqref{eq.psi multiply divide helper}.}

        We have that
        \[\begin{aligned}\frac{\langle x_j, v_z(x_j)\rangle-\langle y_j,v_z(y_j)\rangle}{\|x_j-y_j\|} &= \left\langle \frac{x_j-y_j}{\|x_j-y_j\|}, v_z(y_j)\right\rangle\\
            &\qquad\qquad -\frac{\|v_z(x_j)-v_z(y_j)\|}{\|x_j-y_j\|}\left\langle x_j,\frac{v_z(x_j)-v_z(y_j)}{\|v_z(x_j)-v_z(y_j)\|}\right\rangle,\end{aligned}\]
        where we define $\frac{0}{\|0\|} := 0$. By \cref{prop.daduk interval in Omega}, the map $v_z(\cdot)$ is smooth over $V$ and hence locally upper Lipschitz there. Then since $x\in V$, there exists $C>0$ such that after taking subsequences, the following inequality holds for each $j$:
        \[\frac{\|v_z(x_j)-v_z(y_j)\|}{\|x_j-y_j\|} < C.\]
        Moreover, since $[z]$ is a smooth embedded submanifold of $\mathbb R^d$ and $v_z(x_j),v_z(y_j)\to v_z(x) = z$, we get that $\frac{v_z(x_j)-v_z(y_j)}{\|v_z(x_j)-v_z(y_j)\|} \to t\in T_{z}[z]$. Note further that $x\in N_{z}$ since $z$ is a distance minimizer of $x$ to a neighborhood of $z$ in $[z]$ (\cref{prop.daduk argmax in normal space}). By combining the above observations, we obtain
        \[\left|\frac{\|v_z(x_j)-v_z(y_j)\|}{\|x_j-y_j\|}\left\langle x_j,\frac{v_z(x_j)-v_z(y_j)}{\|v_z(x_j)-v_z(y_j)\|}\right\rangle\right|\leq C\left|\left\langle x_j,\frac{v_z(x_j)-v_z(y_j)}{\|v_z(x_j)-v_z(y_j)\|}\right\rangle\right|\to C\left|\langle x,t\rangle\right| = 0.\]
        By taking subsequences, we have $\frac{x_j-y_j}{\|x_j-y_j\|}\to u_0$ for some $u_0\in\mathbb R^d$ with $\|u_0\|=1$. It follows that
        \begin{equation}
            \label{eq.convergence to vz-u0}
            \frac{\langle x_j, v_z(x_j)\rangle-\langle y_j,v_z(y_j)\rangle}{\|x_j-y_j\|} \to \langle v_z(x),u_0\rangle = \langle z,u_0\rangle.
        \end{equation}
        
        \vspace{\claimstepvspace}\noindent \textit{Step 5. Bounding the left-hand fraction in the right-hand side of \eqref{eq.psi multiply divide helper}.}

        We claim that $1\leq \liminf_{n\to \infty}\frac{\|x_j-y_j\|}{d([x_j],[y_j])}$ is upper bounded. To this end, observe that so far, $i_*\in I$ has been arbitrary and $u_0\in N_x$ does not depend on $i_*$. By definition of $E$, it holds that the set $\{w_i\}_{i\in I}$ is spanning. In particular, there exists $i_0\in I$ such that $\langle w_{i_0},u_0\rangle \neq 0$. Now, since the max filtering map $\llangle[\cdot],[w_{i_0}]\rrangle$ is $\|w_{i_0}\|$-Lipschitz, we obtain that
        \[\left|\frac{\llangle[x_j],[w_{i_0}]\rrangle-\llangle[y_j],[w_{i_0}]\rrangle}{d([x_j],[y_j])}\right| \leq \|w_{i_0}\| <\infty.\]
        On the other hand, by~\eqref{eq.convergence to vz-u0} and since $\langle w_{i_0},u_0\rangle\neq 0$, there exists $d > 0$ such that after taking subsequences, we have
        \[\left|\frac{\langle x_j, v_{w_{i_0}}(x_j)\rangle-\langle y_j,v_{w_{i_0}}(y_j)\rangle}{\|x_j-y_j\|} \right|> d.\] 
       Hence by \eqref{eq.psi multiply divide helper}, we get $1\leq \frac{\|x_j-y_j\|}{d([x_j],[y_j])} \leq \frac{\|w_{i_0}\|}{d} < \infty$ for all $j$, which proves the subclaim of this step.
               
       \vspace{\claimstepvspace}\noindent \textit{Step 6. Finishing the proof.}

       Take a further subsequence so that $\frac{\|x_j-y_j\|}{d([x_j],[y_j])}\to  c_0:=\liminf_{n\to \infty}\frac{\|x_j-y_j\|}{d([x_j],[y_j])}\in [1,\infty)$. The desired convergence \eqref{eq.convergence RG} now follows by taking $u := c_0\cdot u_0$ and combining equations \eqref{eq.psi multiply divide helper} and \eqref{eq.convergence to vz-u0}.
    \end{proof}

\subsubsection{Linear operator injectivity analysis}
The proof of \cref{lem.regular local avoidance} is finished by combining the following claim with all preceding ones.

\begin{claim}
    To obtain the bound \eqref{eq.R3}, it suffices to show that
        \begin{equation}
            \label{eq.E bound}
            \dim(N_x^m - E) \leq mc-1-\left(\left\lceil\frac{m}{\chi(G)}\right\rceil - c\right).
        \end{equation}
        Moreover, the bound above holds.
\end{claim}
\begin{proof}
    By \cref{claim.R3E intersection}, we have that $R_3 \subseteq N_x^m-E$. As such, $\dim(R_3)\leq \dim(N_x^m-E)$, and \eqref{eq.R3} follows from \eqref{eq.E bound}, which we are now left to establish. 

    Since $N_x^m - E$ is a finite union and since the templates $\{z_i\}_{i\notin I}$ are unrestricted in \eqref{eq.E definition}, it suffices to fix $\{h_l\}_{l=1}^m\in (G_x)^m$ and $I\subseteq \{1,\dots,m\}$ with $|I|=\left\lceil\frac{m}{\chi(G)}\right\rceil$ and to show that
    \[\dim(E_1)\leq |I|c-1-\left(|I| - c\right),\]
    where
    \[E_1 :=    \big\{ \{z_i\}_{i\in I}\in N_x^{|I|}: \text{$\{\langle h_iz_i,\cdot\rangle|_{N_x}\}_{i\in I}$ is not injective}\big\}.\]
    Since $N_x$ is $G_x$-invariant (\cref{prop.Nx Tx Gx invariant}) and $\{h_i\}_{i\in I}$ is an isometry of $\mathbb R^I$, it holds that
    \[\dim(E_2) = \dim(\{h_i\}_{i\in I}\cdot E_1) = \dim(E_1),\]
    where 
    \[E_2 = \big\{ \{z_i\}_{i\in I}\in N_x^{|I|}: \text{$\{\langle z_i,\cdot\rangle|_{N_x}\}_{i\in I}$ is not injective}\big\}.\]
    By identifying $N_x$ with $\mathbb R^c$, we lift to the space of singular value decompositions
    \[\begin{aligned} F := &\big\{\big(\{z_i\}_{i\in I},U,\Sigma, W\big)\in  (\mathbb R^c)^{|I|}\times \mathbb R^{c\times (c-1)}\times D_{\geq 0}^{(c-1)\times (c-1)}\times \mathbb R^{|I|\times (c-1)}:\\
        &\qquad\qquad \qquad U^TU = \operatorname{Id}_{c-1} \wedge W^TW = \operatorname{Id}_{c-1}\wedge \{\langle z_i,\cdot \rangle|_{\mathbb R^c}\}_{i\in I} = W\Sigma U^T\big\},
    \end{aligned}\]
    where $D_{\geq 0}^{(c-1)\times (c-1)}\subseteq \mathbb R^{(c-1)\times (c-1)}$ is the cone of diagonal matrices with nonnegative entries. Then $F$ is semialgebraic and $E_2$ is the projection of $F$ onto the first component $\{z_i\}_{i\in I}$. Let $\pi_{\sigma}$ denote the projection onto the other three components $(U,\Sigma,W)$. Then, the fibers of $\pi_\sigma$ are singleton. Hence by \cref{prop.conservation of dimension}, we have $\dim(E_2)\leq \dim(\pi_\sigma(F))$ and it suffices to show that
    \[\dim(\pi_\sigma(F)) = |I|c-1-\left(|I| - c\right).\]
    Observe that
    \[\begin{aligned}\pi_\sigma(F) = \big\{\big(U,\Sigma, W\big)\in & \mathbb R^{c\times (c-1)}\times D_{\geq 0}^{(c-1)\times (c-1)}\times \mathbb R^{|I|\times (c-1)}: U^TU = W^TW = \operatorname{Id}_{c-1}\big\}.
    \end{aligned}\]
    We count dimensions. By the diagonality and orthonormality constraints, it holds that $\Sigma$ has $c-1$ degrees of freedom, $U$ has $c(c-1) - (c-1) - (c-1)(c-2)/2$ degrees of freedom and $W$ has $|I|(c-1)-(c-1)-(c-1)(c-2)/2$ degrees of freedom. The total degrees of freedom are $|I|c - 1 - (|I| - c)$, as desired. This completes the proof of the claim and the lemma.
\end{proof}

\section{Non-differentiability of bilipschitz invariants}
\label{app.nondiff}
In this section, we extend Theorem~21 in~\cite{CahillIM:24} to the case of compact group acting on finite-dimensional real Hilbert spaces. We show that for compact groups, diffentiability of an invariant map at a nonprincipal point forbids it from being lower Lipschitz.

\begin{proposition}
    \label{prop.nondiff}
    Suppose that $G\leq \Oname(d)$ is compact. For any $x\in P(G)^c$ and $G$-invariant map $f\colon \mathbb R^d\to \mathbb R^n$ that is differentiable at $x$, the following statements hold
    \begin{itemize}
        \item[(a)]There exists a unit vector $v\in N_x \cap x^\perp$ such that $Df(x)v=0$.
        \item[(b)]The induced map $f^{\downarrow}\colon \mathbb R^d/G\to \mathbb R^d$ is not lower Lipschitz.
        \item[(c)]If $x\in \mathbb S^{d-1}$, the restriction of the induced map $f^{\downarrow}|_{[\mathbb S^{d-1}]}$ is not lower Lipschitz. 
    \end{itemize}
\end{proposition}
\begin{proof}

    First, we address (a). Since $P(G)$ and $Q_x=G\cdot V_x$ are $G$-invariant, open and dense, we can choose $p\in P(G)\cap V_x$. Then $G_p\leq G_x$ by \cref{prop.stab subset}, and we may pick $g\in G_x\setminus G_p$ since $x$ is not principal. By the same argument as for Theorem~21(a) in~\cite{CahillIM:24}, it suffices to show that $\ker(g-\operatorname{id})^\perp\cap N_x\neq \{0\}$. By \cref{prop.Nx Tx Gx invariant}, we have that $T_x$ and $N_x$ are $g$-invariant. As such, $g$ splits as a block matrix and its $1$-eigenspace $E := \ker(g-\operatorname{id}) = (E\cap T_x)\oplus(E\cap N_x)$ splits into an orthogonal direct sum. It follows that $E^\perp \cap N_x = \{0\}$ if and only if $N_x \subseteq E$, which is false since $p\in V_x \cap E^c\subseteq N_x\cap E^c$ (note that $p\in E^c$ since $g\notin G_p$.)

    For (b) and (c), the same arguments as in the proof of Theorem~21 in~\cite{CahillIM:24} hold as long as we establish that $d([x+tv],[x])=|t|$ for small $t$ and where $v$ is given by (a). This holds since by \cref{lem.UV facts intersection V open in normal}, we have $x+tv \in V_x$ for small $t$.
\end{proof}


\begin{thebibliography}{WW}    

    \bibitem{AlexandrinoB:15}
    \href{https://link.springer.com/book/10.1007/978-3-319-16613-1}{M.\ M.\ Alexandrino, R.\ G.\ Bettiol,
    Lie groups and geometric aspects of isometric actions,
    Springer, 2015.}  
    
    \bibitem{Alharbi:22}
    \href{https://link.springer.com/article/10.1007/s43670-024-00084-y}{W.\ Alharbi, S.\ Alshabhi, D.\ Freeman, D.\ Ghoreishi,
    Locality and stability for phase retrieval,
    Sampl. Theo. Sig. Proc. Dat. Anal. 22(1) (2024) 10.}
    
    \bibitem{AmirGARD:23}
    \href{https://proceedings.neurips.cc/paper_files/paper/2023/hash/84b686f7cc7b7751e9aaac0da74f755a-Abstract-Conference.html}{T.\ Amir, S.\ J.\ Gortler, I.\ Avni, R.\ Ravina, N.\ Dym,
    Neural injective functions for multisets, measures and graphs via a finite witness theorem,
    Adv. Neur. Info. Proc. Sys. 36 (2024).}
    
    \bibitem{AmosXK:17}
    \href{http://proceedings.mlr.press/v70/amos17b.html?ref=https://githubhelp.com}{B.\ Amos, L.\ Xu, J.\ Z.\ Kolter,
    Input convex neural networks,
    Intl.\ Conf.\ ML.\ PMLR (2017) 146-155.}
    
    \bibitem{BalanT:23}
    \href{https://arxiv.org/abs/2308.11784}{R.\ Balan, E.\ Tsoukanis,
    G-invariant representations using coorbits: Bi-lipschitz properties,
    arXiv:2308.11784, 2023.}
    
    \bibitem{BochnakCR:13}
    \href{https://link.springer.com/book/10.1007/978-3-662-03718-8}{J.\ Bochnak, M.\ Coste, M.-F.\ Roy,
    Real algebraic geometry,
    Springer, 2013.}
    
    \bibitem{CahillCD:16}
    \href{https://www.ams.org/journals/btran/2016-03-03/S2330-0000-2016-00012-6/}{J.\ Cahill, P.\ Casazza, I.\ Daubechies,
    Phase retrieval in infinite-dimensional Hilbert spaces,
    Trans.\ Amer.\ Math.\ Soc., Ser.\ B 3 (2016) 63--76.}
    
    \bibitem{CahillIM:24}
    \href{https://www.sciencedirect.com/science/article/pii/S1063520324000460}{J.\ Cahill, J.\ W.\ Iverson, D.\ G.\ Mixon,
    Towards a bilipschitz invariant theory,
    Appl. Comput. Harm. Anal. 72 (2024) 101669.}
    
    \bibitem{CahillIMP:22}
    \href{https://link.springer.com/article/10.1007/s10208-024-09656-9}{J.\ Cahill, J.\ W.\ Iverson, D.\ G.\ Mixon, D.\ Packer,
    Group-invariant max filtering,
    Found. Comput. Math. (2024) 1--38.}
    
    \bibitem{DoCarmo:92}
    \href{https://link.springer.com/book/9780817634902}{M.\ P.\ Do Carmo,
    Riemannian geometry,
    Boston: Birkhäuser, 1992.}
    
    \bibitem{DudekH:94}
    \href{https://bibliotekanauki.pl/articles/1311744}{E.\ Dudek, K.\ Holly,
    Nonlinear orthogonal projection,
    Ann.\ Pol.\ Math.\ 59 (1994) 1--31.}
    
    \bibitem{GallotHL:90}
    \href{https://link.springer.com/book/10.1007/978-3-642-18855-8}{S.\ Gallot, D.\ Hilom, J.\ Lafontaine,
    Riemannian geometry,
    Springer-Verlag, 1990.}
    
    \bibitem{GordoskiL:16}
    \href{https://link.springer.com/article/10.1007/s00208-015-1304-y}{C.\ Gorodski, A.\ Lytchak,
    Isometric actions on spheres with an orbifold quotient,
    Math.\ Ann.\ 365 (2016) 1041--1067.}
    
    \bibitem{JonesOR:11}
    \href{https://www.sciencedirect.com/science/article/pii/S1063520312001078}{P.W.\ Jones, A.\ Osipov, V.\ Rokhlin,
    Randomized approximate nearest neighbors algorithm.
    ONAS\ 108.38 (2011): 15679-15686.}
    
    \bibitem{Lee:13}
    \href{https://link.springer.com/book/10.1007/978-1-4419-9982-5}{J.\ M.\ Lee,
    Introduction to smooth manifolds,
    Springer, 2013.}
    
    \bibitem{Lee:18}
    \href{https://link.springer.com/book/10.1007/978-3-319-91755-9}{J.\ M.\ Lee,
    Introduction to Riemannian manifolds,
    Springer, 2018.}
    
    \bibitem{MixonP:22}
    \href{https://link.springer.com/article/10.1007/s10444-023-10084-6}{D.\ G.\ Mixon, D.\ Packer,
    Max filtering with reflection groups,
    Adv. Comput. Math. 49(6) (2023) 82.}
    
    \bibitem{MixonQ:22}
    \href{https://arxiv.org/abs/2212.11156}{D.\ G.\ Mixon, Y.\ Qaddura,
    Injectivity, stability, and positive definiteness of max filtering,
    arXiv:2212.11156, 2022.}
    
    \bibitem{ZhaoS:14}
    \href{https://www.sciencedirect.com/science/article/abs/pii/S1047847714000549}{Z.\ Zhao, A.\ Singer,
    Rotationally invariant image representation for viewing direction classification in cryo-EM,
    J.\ Struct.\ Bio.\ 186.1 (2014) 153-166.}
    
    \bibitem{ZhaoSS:16}
    \href{https://ieeexplore.ieee.org/abstract/document/7384472}{Z.\ Zhao, Y.\ Shkolnisky, A.\ Singer,
    Fast Steerable Principal Component Analysis,
    IEEE Trans. \ Comput.\ Imag.\ 2.1 (2016) 1-12.}
    
    \end{thebibliography}
\end{document}